\documentclass[reqno,11pt, psfig]{amsart}
\usepackage{cases}
\usepackage{mathrsfs}
\usepackage[T1]{fontenc}
\usepackage{mathrsfs}
\usepackage{amsmath,latexsym,amssymb,amsfonts,amsbsy, amsthm}
\usepackage{bm}
\usepackage[usenames]{color}
\usepackage{xspace,colortbl}
\usepackage{epsfig}
\usepackage{graphicx}
\usepackage{subfigure}
\usepackage{amsmath,amsfonts,amsthm,amssymb,amscd}
\input amssym.def
\input amssym.tex
\usepackage{color}
\usepackage[title,titletoc,toc]{appendix}
\usepackage{bigints}

\allowdisplaybreaks

\newcommand{\be}{\begin{equation} }
	\newcommand{\ee}{\end{equation}}
\newcommand{\bee}{\begin{equation*} }
	\newcommand{\eee}{\end{equation*}}
\newcommand{\bse}{\begin{subequations}}
	\newcommand{\ese}{\end{subequations}}

\usepackage {caption}
\usepackage{xcolor}
\usepackage{cases}

\usepackage{amsmath}
\usepackage{amsthm}

\usepackage{todonotes}
\usepackage{float}
\usepackage{hyperref}

\newtheorem{theorem}{Theorem}[section]
\newtheorem{corollary}[theorem]{Corollary}

\newtheorem{lemma}{Lemma}[section]
\newtheorem{Proposition}{Proposition}[section]
\newtheorem{RHP}{The RH problem}

\theoremstyle{remark}
\newtheorem{remark}{Remark}[section]

\theoremstyle{definition}

\numberwithin{equation}{section}

\title[Global  solution for the  mCH equation]
{On the global existence for  the  modified Camassa-Holm equation via the inverse scattering  method}

\author[]{Yiling YANG}
\address[Yiling Yang]{School of Mathematical Sciences  and Key Laboratory   for Nonlinear Science, Fudan   University, Shanghai 200433, P. R. China.}\email{ylyang19@fudan.edu.cn}
\author[]{Engui Fan}
\address[Engui Fan]{School of Mathematical Sciences  and Key Laboratory   for Nonlinear Science, Fudan   University, Shanghai 200433, P. R. China.}\email{faneg@fudan.edu.cn}
\author[]{Yue Liu}
\address[Yue Liu]{Department of Mathematics, University of Texas at Arlington, TX 76019}\email{yliu@uta.edu}

\date{}                                           % Activate to display a given date or no date
\begin{document}
	%\today
	\thispagestyle{empty}

	\begin{abstract}
		\baselineskip=17pt		
	 In this paper, we address the existence of global solutions to the Cauchy problem of the modified Camassa-Holm (mCH) equation, which is known as a model for the unidirectional propagation of shallow water waves.
 Based on the  spectral analysis of the Lax pair,  we apply the inverse scattering transform to rigorously analyze the mCH equation with zero background. By connecting the Cauchy problem to the  Riemann-Hilbert (RH) problem, we establish a bijective map between potential and reflection coefficients within the $L^2$-Sobolev space framework. Utilizing a reconstruction formula and estimates on the time-dependent RH problem, we obtain a unique global solution to the Cauchy problem for the mCH equation.
	\end{abstract}
	
	\maketitle
	
	\baselineskip=17pt

	\setcounter{tocdepth}{2} \tableofcontents

	\section {Introduction and main result}
	\quad
	In this study, we  investigate  the existence of global solutions to the following Cauchy problem associated with the modified Camassa-Holm (mCH) equation
\begin{align}
	&m_{t}+\left(m\left(u^{2}-u_{x}^{2}\right)\right)_{x}+\kappa u_{x}=0, \quad m=u-u_{x x},\ t>0,\ x\in\mathbb{R},\label{mch}\\
	&m(x, 0)=m_{0}(x), \label{mch1}
\end{align}	
where $ u(t, x)$ is the function in dimensionless space-time variables $(t, x)$,
% with zero background (i.e.  $ m_0(x) \to 0, $ as the spatial variable $ |x| \to + \infty $),
and  $\kappa$ is a positive constant  characterizing  the effect of the linear dispersion.
The mCH equation (\ref{mch}) was first  presented   by Fokas \cite{Fokas} and Fuchssteiner  using  recursion operators \cite{BF1996}, and later  found  by Olver and
Rosenau \cite{PP1996} via   tri-Hamiltonian duality to the bi-Hamiltonian of the mKdV
equation (see also \cite{THFan},  referred to as the Fokas-Olver-Rosenau-Qiao equation). %In \cite{Qiao}, the equation was again derived from the two-dimensional Euler system,   where peakon/cuspon solutions were presented.
%When $\kappa=0$, the equation \eqref{mch} reduces to
%\begin{align}	
%	&m_{t}+\left(m\left(u^{2}-u_{x}^{2}\right)\right)_{x} =0,  \label{mcho}
%\end{align}	
%which   was  derived  by  Fuchssteiner via using  recursion operators \cite{BF1996}.
%The mCH equation \eqref{mcho}
%is also   referred to as the Fokas-Olver-Rosenau-Qiao equation  \cite{THFan},
%but is mostly known as the  mCH  equation.
Recently  the mCH equation \eqref{mch}  was considered as a model for the unidirectional propagation for
shallow-water waves of mild amplitude over a flat bottom \cite{CHL}, where the solution $ u $ is related to  the horizontal velocity in certain level of water, and $\kappa>0$  is a parameter related to the
critical shallow water speed.  In the short-wave limit case,
 the mCH equation \eqref{mch} reduces to the short-pulse equation \cite{qu7}
 \begin{align*}
 	v_{xt}=\frac{1}{3}(v^3)_{xx}+\kappa v,
 \end{align*}
which is a model for the propagation of ultra-short light pulses in silica optical fibers \cite{Sch2004} and it is also an approximation of nonlinear wave packets in dispersive media in the limit of few cycles on the ultra-short pulse scale \cite{Sch2005}. Moreover, the  short-pulse equation and the nonlinear Schr\"odinger equation are both  derived from Maxwell's equations, but the numerical simulations in \cite{Sch2005} shows that  as the pulse length shortens, the  nonlinear Schr\"odinger equation  approximation becomes steadily less accurate, while the short-pulse equation provides a better and better approximation.

It is noted that the    mCH equation in \eqref{mch} for any $ \kappa \geq 0$ is  completely  integrable  and admits  the Lax pairs \cite{Qiao2,Sch96}.
%It is noted that  the    mCH equations in \eqref{mch} for any $\kappa \geq 0$ is  completely  integrable  and admit  the Lax pairs \cite{Qiao2,Sch96}.
However, unlike the Camassa-Holm (CH) equation \cite{Holm1,cola,Fuchssteiner1} and the Degasperis-Procesi (DP) equation \cite{DP}, the mCH equation \eqref{mch}  with a positive parameter $\kappa$ cannot be transformed into the mCH equation \eqref{mch} with $\kappa=0$ using the Galilean transformation.

%  The mCH equation  \eqref{mch} admits   nonsmooth soliton solutions       were obtained   by the method of inverse scattering  transformation method \cite{Qiao2}.
%Wave-breaking, peakons and local well-posedness  to the Cauchy problem  for the mCH equation in the Sobolev
%space $H^{s}$, $s>5/2$  were  given  \cite{qu7}.  %They
%deactivated that   the mCH equation can be regarded as a Euclidean-invariant version of the CH equation, just as the mKdV equation is a Euclidean-invariant counterpart to the KdV equation from the viewpoint of curve flows in Klein geometries.
The local well-posedness  to the Cauchy problem  associated with  the mCH equation \eqref {mch} for the initial profile $ m_0 \in H^{s}$, $s>1/2$ was  established  in \cite{qu7}. A blow-up mechanism to the mCH equation  \eqref{mch} is provided in \cite{qu8}. However, unlike the CH equation case in \cite{CA2000}, even if the initial potential $m_0$ does not change sign, the second derivative item $u_{xx}$ in the equation  \eqref{mch} may also blow-up in finite time. On the other hand, the mCH equation  \eqref{mch} admits   the smooth soliton solutions \cite{Matsuno2} as the wave speed $ c \ge 2\kappa > 0  $ and   the orbital  stability of those smooth solitons
 in the Sobolev spaces $H^1(\mathbb{R})\cap W^{1,4}(\mathbb{R})$ and the spectral stability of such smooth
solitons are obtained  without
the condition of positive Randon measure  \cite{LL2021}.

It is worth mentioning that the behavior of  solitons with $ \kappa = 0 $ is distinct from that of the smooth solitons ($\kappa > 0 $) of the mCH equation (\ref{mch}). This leads to new kinds of singular solitons known as peakons \cite{qu8}. In this case, $\kappa = 0, $ it was  shown that  the    mCH equation  \eqref{mch} with  a nonzero background (i.e. the initial value  $ m_0(x) \to \gamma \neq 0, $ as $ |x| \to \infty$)
may  support   smooth soliton solutions by the  B\"acklund transformation \cite{Matsuno1} and the Riemann-Hilbert(RH) method \cite{Mon}.
Recently,     the     long time   asymptotic behavior of the mCH equation  \eqref{mch}  with $\kappa > 0 $    was obtained by using $\bar\partial$-steepest decadent  analysis \cite{YYLmch}.
The existence of global solutions to the Cauchy problem for the
mCH equation \eqref{mch} with $ \kappa = 0 $ and a nonzero background   was studied  \cite{YFL}.

%In the present paper, we study the   bijectivity for the scattering-inverse scattering transforms associated with the  mCH equation (\ref{mch}) and global well-posedness  for the initial value problem of  the mCH  equation (\ref{mch}).

The primary objective of this paper is to utilize the $L^2$-Sobolev space bijectivity between potential and reflection coefficient to establish the global well-posedness of the Cauchy problem for the mCH equation (\ref{mch}). The main result is stated as follows (the proof framework is described in Figure \ref{Figure1}).

%The aim of present paper is to establish  the  $L^2$-Sobolev space  bijectivity  between potential and reflection coefficient
%and then further  prove  the  global well-posedness  for the initial value problem of  the mCH  equation (\ref{mch}).
%The  main result  is stated  as follows (its proof frame   see Figure \ref{Figure1}).
\begin{theorem}\label{last}  Assume that  $m_0 \in   H^{2,1}(\mathbb{R})$ is sufficiently small. Then	there exists a unique global solution   $m \in C([0, +\infty),  H^{2,1}(\mathbb{R}))$   to  the Cauchy  problem  (\ref{mch})-(\ref{mch1}) with the initial value $ u(0, x) = m_0(x),  \; \forall x \in \mathbb{R}. $   Furthermore, the map
	$$m_0  \in  H^{2,1}(\mathbb{R})\longrightarrow m \in C([0, +\infty),  H^{2,1}(\mathbb{R}))$$
	is locally Lipschitz continuous.
\end{theorem}
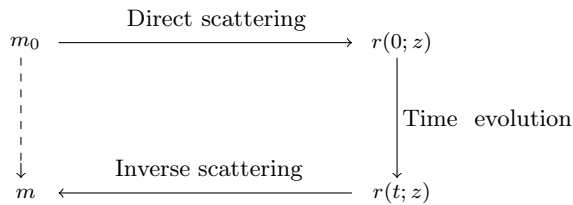
\begin{figure}
	\begin{center}
		\begin{tikzpicture}
				\node at (-3.0, 1.0 ){\fontsize{8pt}{8pt} $m_0$};
					\draw [->] (-2.5,1.0)--(1.4,1.0);
					\node at (2, 1.0 ){\fontsize{8pt}{8pt} $r (0;z) $};		
				\node at (-3.0, -1.0 ){\fontsize{8pt}{8pt} $m  $};
					\draw [<-] (-2.5,-1.0)--(1.4,-1.0);
					\node at (2, -1.0 ){\fontsize{8pt}{8pt} $r (t;z)$};
			
						\draw [-] (2,0.8)--(2,-0.6);
					\draw [->] (2,-0.6)--(2,-0.8);
						\draw [dashed] (-3,0.8)--(-3,-0.6);
						\draw [->] (-3,-0.6)--(-3,-0.8);
						\node at (3.2, 0 ){\footnotesize Time \ evolution};
						\node at (-0.4, 1.3 ){\footnotesize Direct scattering};
						\node at (-0.5, -0.7){\footnotesize Inverse scattering};
					\end{tikzpicture}
				\caption{\footnotesize   Bijective relation between potential $m$ and reflection coefficients }
				\label{Figure1}
			\end{center}
		\end{figure}

The key tool used to prove the result above is the inverse scattering theory, initially developed for the first-order spectral problem, and later applied to demonstrate global existence in integrable systems within the Schwartz space \cite{BC1984, BC1985}. Moreover, Zhou provided a rigorous framework to solve the Cauchy problem in weighted $L^2$-Sobolev spaces through solvability analysis of the RH problem in \cite{Zhou1989, ZhouL2}.
Recently, this approach has found extensive application in proving the global well-posedness of integrable equations, including the nonlinear Schr\"odinger (NLS), derivative NLS, and mCH equations with a nonzero background  \cite{Deift2011, RN10, LiuDNLS, YFL, pelinovskey1}.

%The key tool to prove our above result is  the inverse scattering theory, which
%was first developed   for the first-order   spectral problem and
%was   used to prove   global  existence    for integrable  systems in the Schwartz space \cite{BC1984,BC1985}. Further, via solvability analysis  of the RH problem in \cite{Zhou1989}, Zhou provided a rigorous 	%framework to solve the Cauchy problem in weighted $L^2 $-Sobolev spaces \cite{ZhouL2}.
%	In recent years,  this method has  been  widely applied to prove  the  global well-posedness of integrable equations, such as nonlinear Schr\"odinger (NLS),  derivative NLS, mCH equations 	 %\cite{Deift2011,RN10,LiuDNLS,YFL,pelinovskey1}.

We would like to point out that the assumption $m_0 \in   H^{2,1}(\mathbb{R})$ made  in Theorem
\ref{last} is excluded from the blow-up condition in \cite{qu8}. Furthermore,  compared with  our previous  results in \cite{YFL}, extending the inverse scattering transform approach to the mCH equation (\ref{mch}) presents certain distinctions and difficulties. The most significant distinction arises from the fact that the RH problem corresponding to the mCH equation (\ref{mch}) lacks an explicit symmetry expression between $z$ and $-1/z$, leading to two substantial difficulties:

\begin{itemize}
	\item[(1)] In contrast to the derivative NLS equation \cite{liu, pelinovskey1}, the transformation of the Jost function from the spectral parameter $z$ to $k=z-\frac{1}{z}$ poses challenges in constructing the RH problem. As a result, direct application of the Fourier transformation for estimating the RH problem is not feasible. However, this obstacle is successfully addressed by leveraging the special structure of the phase function $\theta(z)$ defined in \eqref{theta} and exploiting the symmetry of the reflection coefficient between $z$ and $-1/z$ (see Lemma \ref{lemmaFourier} in Section \ref{sec62}).
	
	\item[(2)] The eigenvalues and resonances associated with the Lax pair \eqref{lax0} may be located anywhere in $\mathbb{C}^+\cup \mathbb{R}$ and can be non-simple, leading to challenges in controlling the norm of the solution to the RH problem. Notably, the eigenvalues on $\mathbb{C}^+$ only affect the estimation of the solution from the RH problem in Section \ref{sec3} but do not impact its unique solvability.
	On the other hand, the presence of the spectrum on $\mathbb{R}$, known as resonance, introduces difficulties in estimating the RH problem. Research on resonance necessitates alternative methods for estimation, as exemplified in \cite{Zhou1989CPAM, Jenkins}. To overcome these challenges, a small norm condition on the initial value $m_0$ becomes necessary to ensure the absence of eigenvalues and resonance (see Subsection \ref{subec r}).

\end{itemize}

The rest of the paper is structured as follows: In Section \ref{sec2}, we delve into the direct scattering transform, which maps the initial data $m_0(x)$ to reflection coefficient. In Section \ref{sec3}, we establish two RH problems associated with the mCH equation \eqref{mch}. We rigorously demonstrate their solvability in the space $H^{s}(\mathbb{R}_k)$, where $s>1/2$, and obtain estimates on their solutions. In Section \ref{sec6}, we analyze the time-evolution of reflection coefficient and the RH problem to present the proof of our main result, Theorem \ref{last}.

\vskip 0.2cm

\noindent {\bf Notations.}  We now introduce some notations used this paper.
The classical Pauli matrices $\{\sigma_j\}_{j=1,2,3}$ are defined by
\begin{equation}\label{def:PauliM}
	\sigma_1:=\begin{pmatrix}0 & 1 \\ 1 & 0\end{pmatrix}, \quad
	\sigma_2:=\begin{pmatrix}0 & -i \\ i & 0\end{pmatrix}, \quad
	\sigma_3:=\begin{pmatrix}1 & 0 \\ 0 & -1\end{pmatrix}.
\end{equation}
	If $I$ is an interval on the real line $\mathbb{R}$ and $\mathrm{X}$ is a  Banach space, then $C_b(I,\mathrm{X})$ denotes the space of  bounded continuous functions on $I$ taking values in $\mathrm{X}$. It is equipped with the norm
	\begin{equation*}
		\|f\|_{C_b(I, \mathrm{X})}=\sup _{x \in I}\|f(x)\|_{\mathrm{X}}.
	\end{equation*}
	%Moreover, denote $C^0_B(\mathrm{X})$ as a  space of bounded continuous functions on $\mathrm{X}$.
	
	%If the  entries  $f_1$ and $f_2$  are in space $\mathrm{X}$,  then we call vector  $\vec{f}=(f_1,f_2)^T$  is in space $\mathrm{X}$ with  the norm $\parallel \vec{f}\parallel_\mathrm{X}\triangleq \parallel f_1\parallel_\mathrm{X}+\parallel f_2\parallel_\mathrm{X}$. Similarly, if every  entries of  matrix $A$ are in space $\mathrm{X}$, then we call $A$ is also in space $\mathrm{X}$.
	
	We introduce  the normed spaces:
	\begin{itemize}
		\item A weighted $L^p(\mathbb{R})$ space is specified by
		$$L^{p,s}(\mathbb{R})  =  \left\lbrace f(x)\in L^p(\mathbb{R}) | \hspace{0.1cm} \langle \cdot\rangle^sf(x)\in L^p(\mathbb{R}) \right\rbrace; $$
		
		\item A weighted Sobolev space  is defined by
		$$H^{l,s}(\mathbb{R})   =  \left\lbrace f(x)\in L^2(\mathbb{R}) | \hspace{0.1cm} \langle \cdot\rangle^s\partial^jf(x)\in L^2(\mathbb{R}),  \text{ for }j=1,...,l \right\rbrace.$$
		
	\end{itemize}
	For the simplicity,  the norm of $f(x)\in L^{p}(\mathbb{R})$ and $g(x)\in L^{p,s}(\mathbb{R})$ are  abbreviated to $\parallel f\parallel_{p}$,  $\parallel g\parallel_{p,s}$ respectively.
	If a fucntion $f(z)$, $z\in\mathbb{R}$ admits the symmetry $f(z)=f(-1/z)$, then for $k=z-1/z$,
	\begin{align*}
		g(k)=f(z(k))
	\end{align*}
	is a well defined function on $\mathbb{R}$. We call $f\in H^{l,s}(\mathbb{R}_k) $ if $g\in H^{l,s}(\mathbb{R})$.
	
	Finally,  the letter $C$ will be used to denote universal positive constants which may vary from line to line. %We also use the	notation $A\lesssim B$ to denote the bound of the form $A \leq CB$.
	 To emphasize the implied constant
	to depend  on some parameter $\alpha$, we shall indicate this by $C(\alpha)$.
	
	\section {Direct scattering transform}\label{sec2}
	
	In this section, we provide the framework  on  the spectral analysis of  the Lax pair to establish the direct scattering transform and address the RH problem
	with initial data  $m_0\in H^{2,1}(\mathbb{R})$.

	\subsection{Spectral analysis on the Lax pair}
	
	For convenience,    we fix $\kappa=2$  in the mCH equation \eqref{mch}  without  loss of generality,   since under a  simple  transformation
	\begin{align}
		x=\tilde{x}, \ \
		t=\frac{2}{\kappa} \tilde{t},  \ \
		u(x, t)=\sqrt{\frac{\kappa}{2}} \tilde{u}(\tilde{x}, \tilde{t}),
	\end{align}
	the  equation (\ref{mch}) becomes
	\begin{equation}
		\tilde{m}_{\tilde{t}}+\left(\tilde{m}\left(\tilde{u}^{2}-\tilde{u}_{\tilde{x}}^{2}\right)\right)_{\tilde{x}}+2 \tilde{u}_{\tilde{x}}=0, \quad \tilde{m}=\tilde{u}-\tilde{u}_{x x}.
	\end{equation}
	
	Then the mCH equation (\ref{mch})   with $\kappa=2$     admits the Lax pair \cite{Qiao2,Sch96}
	\begin{equation}
		\Phi_x = X \Phi,\hspace{0.5cm}\Phi_t =T \Phi, \label{lax0}
	\end{equation}
	where
	\begin{equation}
		X=-\frac{ik}{4}\sigma_3+\frac{i\lambda m(x,t)}{2}\sigma_2,\nonumber
	\end{equation}
	\begin{equation}
		T=\frac{ik}{2\lambda^2}\sigma_3+\frac{ik}{4} \left(u^{2}-u_{x}^{2}\right)\sigma_3-i\left(\frac{2iu-k u_{x}}{2\lambda}+\frac{\lambda}{2} \left(u^{2}-u_{x}^{2}\right) m \right) \sigma_2
		\nonumber
	\end{equation}
	with
	$$k=k(z)=z-\frac{1}{z},\hspace{0.3cm}   \lambda= \lambda(z)=\frac{1}{2}(z+\frac{1}{z}),$$
	and  $z\in \mathbb{C}$ is  a spectral parameter.
	
	Before making  the  direct scattering transform,
	we first write  equation (\ref{mch})  and  its Lax pair into  the  forms  in   a new space variable $y$.
	From the conservation law of the equation  \eqref{mch},
	\begin{align*}
		q_t+(q(u^2-u_x^2))_x=0,
	\end{align*}
	we  introduce  the coordinate transform
	\begin{align}
		dy=qdx-q(u^2-u_x^2)dt,
	\end{align}
	and obtain a reciprocal    transformation
	\begin{equation}
		y(x,t)=x-\int_{x}^{+\infty} \left(q(s,t)-1\right) ds,\label{transy}
	\end{equation}
	where
	\begin{align}
		q(x,t)=\sqrt{m(x,t)^2+1}.
	\end{align}
	Then under  variables $(y,t)$,
	 the mCH equation \eqref{mch} reads
	\begin{align}
			q_t+2u_y	m=0,\label{mchy}
	\end{align}
	and the Lax pair \eqref{lax0} becomes
	\begin{align}
		&\Phi_y = Y \Phi,\hspace{0.5cm} \Phi_t = Q \Phi\label{lax y}\\
		&	Y=\frac{1}{\sqrt{m^2+1}}X,\hspace{0.5cm}Q=(u^2-u_x^2)X+T.\label{laxy}
	\end{align}
	
	In the direct scattering transform, we first consider   the partial spectral problem in the Lax pair  \eqref{lax0} under $t=0$, which  is omitted as usual. For example, $\Phi(z;0,y)$ is just written as $\Phi(z;y)$ or $\Phi$ for simplicity. In addition,	under $m_0\in H^{2,1}$,    the map $x\to y(x,0)$ is continuous differentiable bijection.

	Define a transformation
	\begin{equation}
		\mu(z)\triangleq\mu(z;y)=F(y)^{-1}\Phi(z;y) e^{\frac{i}{4}(z-\frac{1}{z})y\sigma_3}\label{transmu},
	\end{equation}
	where
	\begin{align}
		F(y)=\sqrt{\frac{q_0+1}{2 q_0}}\left(\begin{array}{cc}
			1 & \frac{-i m_0}{q_0+1} \\
			\frac{-i m_0}{q_0+1} & 1
		\end{array}\right)\label{F}.
	\end{align}
	It is thereby inferred that
	\begin{align}
		&\mu_y = -\frac{i}{4}(z-\frac{1}{z})[\sigma_3,\mu]+P\mu,\label{lax1.1}
	\end{align}
	where  the Lie bracket $ [A, B] $ is defined by $ [A, B] = AB -BA  $ and
	\begin{align*}
		&P=\frac{i m_{0,x}}{2q_0^3}\sigma_1+  \frac{m_0}{2 z q_0^2}\left(\begin{array}{cc}
			-i m_0 & 1 \\
			-1 & i m_0
		\end{array}\right).
	\end{align*}
	The equation \eqref{lax1.1} leads to two  Volterra type integrals
	\begin{equation}
		\mu_\pm(z;y)=I+\int^{y}_{\pm \infty}e^{-\frac{i}{4}(z-\frac{1}{z})(y-s)\hat{\sigma}_3}P(s,z)\mu_\pm(z;s)ds\label{intmu}.
	\end{equation}
	Then $ \mu_\pm (z)$ admit two kinds of  reduction conditions
	\begin{equation}
		\mu_\pm(z)=\sigma_2\overline{\mu_\pm(\bar{z})}\sigma_2=\sigma_1\mu_\pm(-z)\sigma_1\label{symPhi1}
	\end{equation}
	and
	\begin{equation}
		\mu_\pm(z)=F^{-2}\sigma_2\mu_\pm(-z^{-1})\sigma_2.\label{symPhi2}
	\end{equation}
	Denote the matrix in column
	$$\mu_\pm(z;y)=\left(  \mu_{\pm,1}(z;y),  \mu_{\pm,2}(z;y)\right), $$
	where  the subscript $1$ and $ 2$ indicate
	the first and second columns of $\mu_\pm(z;y)$, respectively. The integral property of $\mu_\pm(z;y)$ is given in the following proposition.
	\begin{Proposition}\label{Propmu}
		Let $m_0\in H^{2,1}(\mathbb{R})$. Then $\mu_\pm(z;y)$  uniquely exist   such that $\mu_\pm(z;y)-I\in L^{\infty}(\mathbb{R}^\pm\times\mathbb{R})\cap C_b(\mathbb{R}^\pm, L^2(\mathbb{R}))\cap L^{2}(\mathbb{R}^\pm\times\mathbb{R})$, $\partial_z\mu_\pm(z;y),\ k(z)\partial_z\mu_\pm(z;y)\in C_b(\mathbb{R}^\pm, L^2(\mathbb{R}))$. Furthermore, the non-diagonal terms of $k(z)(\mu_\pm(z;y)-I)$ belong in $ C_b(\mathbb{R}^\pm, L^{2}(\mathbb{R}))$.
	\end{Proposition}
	\begin{proof}
		Invoking the definition of variable $y$ in \eqref{transy}, $m_0\in H^{2,1}(\mathbb{R}_x)$ is equivalent to $m_0\in H^{2,1}(\mathbb{R}_y)$.
		From the symmetry of $\mu_\pm$ in \eqref{symPhi2}, it is readily seen  that
		$$\sup_{|z|\leq1}|\mu_\pm(z)|=\sup_{|z|\geq1}|F(y)^{-2}\sigma_2\mu_\pm(-z^{-1})\sigma_2|,$$
		 which implies that
		\begin{align*}
			&\inf_{y\in\mathbb{R}}|F(y)^{-2}|\sup_{|z|\geq1}|\mu_\pm(z)| 	\leq \sup_{|z|\leq1}|\mu_\pm(z)|\leq \|F^{-2}\|_\infty\sup_{|z|\geq1}|\mu_\pm(z)|,\\
			&\int_{|z|\leq1}|z^{-j}\mu_\pm(z)|^2dz\leq \int_{|z|\geq1}\|F^{-2}\|_\infty^2z^{2j-2}|\mu_\pm(z)|^2dz,\ j=0,1,2.
		\end{align*}
		Thus, it is sufficient to analyze the integral equation on \eqref{intmu} spaces $L^{\infty}(\mathbb{R}^\pm\times\{|z|\geq1\})$, $ C_b(\mathbb{R}^\pm, L^2(\{|z|\geq1\}))$ and $ L^{2}(\mathbb{R}^\pm\times\mathbb{R})$ which are abbreviated to $L^\infty$, $ C_b$, $ L^2 $ respectively.  We  denote
		\begin{align}
			n(y,z)=	\mu_{+,1} (z;y)-e_1 ,\label{n}
		\end{align}
		where $e_1=(1,0)^T$.   Introduce the integral operator $\mathcal{T}$
		\begin{align}\label{T1}
			\mathcal{T}(f)(y,z)=-\int_{y}^{+\infty}K(y,s,z)f(s,z)ds,
		\end{align}
		where the integral kernel $K(y,s,z)$ is
		\begin{align}\label{K}
			K(y,s,z)=&
			\frac{i m_{0,x}(s)}{2q_0^3(s)}\left(\begin{array}{cc}
				0 & 1 \\
				e^{\frac{i}{2}(z-1/z)(y-s)} & 0
			\end{array}\right)\nonumber\\
			&+\frac{1}{2 z} \frac{m_0(s)}{q_0^2(s)}\left(\begin{array}{cc}
				-i m_0(s) & 1 \\
				-e^{\frac{i}{2}(z-1/z)(y-s)} & e^{\frac{i}{2}(z-1/z)(y-s)}i m_0(s)
			\end{array}\right),\ x<s.
		\end{align}
		Then the first column of \eqref{intmu} is changed   to
		\begin{align}\label{eqn}
			&n=\mathcal{T}(e_1)+\mathcal{T}(n),
		\end{align}
		where
		\begin{align*}
			\mathcal{T}(e_1)=-\int_{y}^{+\infty}\left(\begin{array}{cc}
				\frac{-im_0^2(s)}{2zq_0^2(s)} \\
				e^{\frac{i}{2}(z-1/z)(y-s)}\left(	\frac{i m_{0,x}(s)}{2q_0^3(s)}- \frac{m_0(s)}{2 zq_0^2(s)} \right)
			\end{array}\right)ds,
		\end{align*}
		with estimates
		\begin{align*}
			&|\mathcal{T}(e_1)|\leq (\|m_0\|_2^2+\|m_0\|_1+\|m_{0,x}\|_1)/2,\\
			&\int_{|z|\geq1}|\mathcal{T}(e_1)|^2dz\leq C(\int_y^{+\infty}|m_0|^2ds+\int_y^{+\infty}|m_{0,x}|^2ds),\\
			&\int_{\mathbb{R}^+}\int_{|z|\geq1}|\mathcal{T}(e_1)|^2dzdy\leq C (\|m_0\|_{2,1/2}+\|m_{0,x}\|_{2,1/2}).
		\end{align*}
		It follows from Lemma 5 in  \cite{YYLmch} that 	the integral operator $\mathcal{T}$ maps $L^\infty\cap C_b\cap L^2 $ to itself with
		\begin{align*}
			\|\mathcal{T}\|\leq C(\|m_0\|_{H^{1,1}}).
		\end{align*}
		In addition, $(I-\mathcal{T})^{-1} $ exists as a bounded	operator on  $L^\infty\cap C_b\cap L^2 $ also admitting
		\begin{align*}
			\|(I-\mathcal{T})^{-1}\|\leq C(\|m_0\|_{H^{1,1}}).
		\end{align*}
		Thereby we conclude that $n=(I-\mathcal{T})^{-1}\mathcal{T}(e_1)\in L^\infty\cap C_b\cap L^2 $. On the other hand, by denoting $\mathcal{T}_z$ as a integral operator with the integral kernel
		\begin{align*}
			\partial_zK(y,s,z)=&
			\frac{-m_0(s)}{2 z^2q_0^2(s)}\left(\begin{array}{cc}
				-i m_0(s) & 1 \\
				0 & 0
			\end{array}\right)-
			\frac{ (1+1/z^2)(y-s)m_{0,x}(s)e^{\frac{i}{2}(z-1/z)(y-s)}}{4q_0^3(s)}\left(\begin{array}{cc}
				0 & 0 \\
				1 & 0
			\end{array}\right)\nonumber\\
			&- \frac{m_0(s)e^{\frac{i}{2}(z-1/z)(y-s)} }{4 zq_0^2(s)}(2/z-i(1+1/z^2)(y-s))\left(\begin{array}{cc}
				0 & 0 \\
				-1 & i m_0(s)
			\end{array}\right),\ x<s,
		\end{align*}
		take $z$-derivative in \eqref{eqn} and obtain
		\begin{align}
			&\partial_zn=\partial_z\mathcal{T}(e_1)+\mathcal{T}_z(n)+T(\partial_zn),
		\end{align}
		where
		\begin{align*}
			\partial_z\mathcal{T}(e_1)=\int_{y}^{+\infty}\left(\begin{array}{cc}
				\frac{-im_0^2(s)}{2z^2q_0^2(s)} \\
				e^{\frac{i}{2}(z-1/z)(y-s)}\left(\frac{m_0(s)}{2 z^2q_0^2(s)}+	\frac{i}{2}(1+1/z^2)(y-s)\left( \frac{i m_{0,x}(s)}{2q_0^3(s)}- \frac{m_0(s)}{2 zq_0^2(s)}\right)  \right)
			\end{array}\right)ds,
		\end{align*}
		with the estimate
		\begin{align*}
			\|\partial_z\mathcal{T}(e_1)\|_{C_b}\leq C(\|m_0\|_{H^{1,1}}).
		\end{align*}
		Noting that for any functions $f(y),\ g(y,z)$, it is found that
		\begin{align*}
			&\int_{|z|\geq1}\big|\int_{y}^{+\infty}f(s)g(s,z)ds\big|^2dz\leq \|f\|_2\|g\|_{2}^2,\\ &\text{or }	\int_{|z|\geq1}\big|\int_{y}^{+\infty}f(s)g(s,z)ds\big|^2dz\leq \|f\|_1\|g\|_{C_b}.
		\end{align*}
		Thus,
		\begin{align*}
			\|\mathcal{T}_zn\|_{C_b}\leq C(\|m_0\|_{H^{1,1}}) \|n\|_{2},
		\end{align*}
		which shows that $\partial_zn$ exists in $C_b$ with
		\begin{align}
			\partial_zn=(I-\mathcal{T})^{-1}(\mathcal{T}_zn+\mathcal{T}(e_1)).\label{partzn}
		\end{align}
		Denote $n=(n_1,n_2)^T$. Via integration by parts, it is adduced that
		\begin{align}
			\frac{i}{2}(z-1/z)n_2&\triangleq I_1+I_2+I_3,\label{kn2}
		\end{align}
		where
		\begin{align*}
			&I_1=\int^y_{+\infty}e^{\frac{i}{2}(z-1/z)(y-s)}\partial_s\left( \frac{i m_{0,x}}{2q_0^3}-\frac{m_0}{2zq_0^2}\right) ds\\
			&I_2=\int^y_{+\infty}e^{\frac{i}{2}(z-1/z)(y-s)}\left[\partial_s\left( \frac{i m_{0,x}}{2q_0^3}-\frac{m_0}{2 zq_0^2}\right) n_1+\partial_s\left( \frac{im_0^2}{2zq_0^2}\right) n_2\right] ds\\
			&I_3=\int^y_{+\infty}e^{\frac{i}{2}(z-1/z)(y-s)}\left[\left( \frac{i m_{0,x}}{2q_0^3}-\frac{m_0}{2 zq_0^2}\right) \partial_sn_1+\frac{im_0^2}{2 zq_0^2} \partial_sn_2\right] ds.
		\end{align*}
		The $C_b$ norm of the $I_1$ is controlled by $\|\partial_s\left( \frac{i m_{0,x}}{2q_0^3}-\frac{m_0}{2q_0^2}\right)\|_2$, namely, $\|m_0\|_{H^{2}}$.
		And for $I_2$, it follows that
		\begin{align*}
			\|I_2\|_{C_b}\leq C(\|m_0\|_{H^{2,1}}) \|n\|_{2}.
		\end{align*}
		From the definition of $n$ in \eqref{n} and the equation \eqref{lax1.1}, it holds that
		\begin{align*}
			\partial_sn=\frac{i}{2}(z-1/z)\left(\begin{array}{cc}
				0 \\
				n_2
			\end{array}\right)+\left(\begin{array}{cc}
				\frac{-im_0^2}{2 zq_0^2}(n_1+1)+\left( \frac{i m_{0,x}}{2q_0^3}+ \frac{m_0}{2 zq_0^2}\right) n_2 \\
				\left(	\frac{i m_{0,x}}{2q_0^3}-\frac{m_0}{2 zq_0^2} \right)  (n_1+1)+\frac{im_0^2}{2 zq_0^2}n_2
			\end{array}\right).
		\end{align*}
		Therefore, $I_3$ can be rewritten as
		\begin{align*}
			&I_3=\int^y_{+\infty}e^{\frac{i}{2}(z-1/z)(y-s)}\left[\frac{-im_0^2}{4zq_0^2}\left( \frac{i m_{0,x}}{q_0^3}-\frac{m_0}{ zq_0^2}\right)-  \frac{m_0^4}{4z^2q_0^4} \right] ds\\
			&+\int^y_{+\infty}e^{\frac{i}{2}(z-1/z)(y-s)}\left( \frac{i m_{0,x}}{q_0^3}-\frac{m_0}{ zq_0^2}\right)\left(\frac{-im_0^2}{2 zq_0^2}n_1+\left( \frac{i m_{0,x}}{2q_0^3}+ \frac{m_0}{2 zq_0^2}\right) n_2 \right) ds\\
			&+\int^y_{+\infty}e^{\frac{i}{2}(z-1/z)(y-s)}\frac{im_0^2}{2 q_0^2}\left[ \frac{i(1-1/z^2)}{2}n_2+\left(	\frac{i m_{0,x}}{2zq_0^3}-\frac{m_0}{2 z^2q_0^2} \right)  n_1+\frac{im_0^2}{2 z^2q_0^2}n_2\right] ds.
		\end{align*}
		Thus it is inferred  that
		\begin{align*}
			\|I_3\|_{C_b}\leq C(\|m_0\|_{H^{2,1}}) \|n\|_{2}.
		\end{align*}
		For the term $\partial_zn$, obviously,
		\begin{align*}
			\|	k(\cdot)(\partial_z\mathcal{T}(e_1))_1\|_{C_b}\leq \|m_0\|_2,
		\end{align*}
		and
		\begin{align*}
			&k(z)(\partial_z\mathcal{T}(e_1))_2=\int_{y}^{+\infty}
			e^{\frac{i}{2}(z-1/z)(y-s)}\frac{i}{2}(1-1/z^2)\frac{m_0(s)}{2 zq_0^2(s)}ds\nonumber\\
			&+\int_{y}^{+\infty}
			e^{\frac{i}{2}(z-1/z)(y-s)}\frac{(z-1/z)(1+1/z^2)(y-s)}{4}\left( \frac{i m_{0,x}(s)}{2q_0^3(s)}- \frac{m_0(s)}{2 zq_0^2(s)}\right)  ds.
		\end{align*}
		Similarly, the $C_b$ norm of  first integral in the right side of above equation is controlled by  $\|m_0\|_1$, and  from integration by parts,  the second integral in the right side of above equation is controlled by $\|m_0\|_{H^{2,1}}$. And
		\begin{align}
			&k(z)\mathcal{T}_zn=\int_{y}^{+\infty}\frac{i}{2}(1-1/z^2)	\frac{-m_0(s)}{2 zq_0^2(s)}\left(\begin{array}{cc}
				-i m_0(s) & 1 \\
				-e^{\frac{i}{2}(z-1/z)(y-s)} & i m_0(s)e^{\frac{i}{2}(z-1/z)(y-s)}
			\end{array}\right)n(s,z)ds\nonumber\\
			&- \int_{y}^{+\infty}\frac{i}{2}(1-1/z^2)\frac{m_0(s)e^{-i\frac{i}{2}(z-1/z)(y-s)}(1+1/z^2)(y-s) }{4 q_0^2(s)}\left(\begin{array}{cc}
				0 & 0 \\
				-1 & i m_0(s)
			\end{array}\right)n(s,z)ds\nonumber\\
			&+\int_{y}^{+\infty}\frac{i}{2}(z-1/z)
			\frac{ -(1+1/z^2)(y-s)m_{0,x}(s)e^{\frac{i}{2}(z-1/z)(y-s)}}{4q_0^3(s)}\left(\begin{array}{cc}
				0 & 0 \\
				1 & 0
			\end{array}\right)n(s,z)ds.\label{knz2}
		\end{align}
		The $C_b$ norm of  first two integrals in the right side of above equation are controlled by  $\|m_0\|_{2,1}$. For the last integral,
		integration by parts gives that
		\begin{align*}
			&\int_{y}^{+\infty}\frac{i}{2}(z-1/z)
			\frac{ -(1+1/z^2)(y-s)m_{0,x}(s)e^{\frac{i}{2}(z-1/z)(y-s)}}{4q_0^3(s)}n_1(s,z)ds\nonumber\\
			=&\int_{y}^{+\infty}(1+1/z^2)e^{\frac{i}{2}(z-1/z)(y-s)}\partial_s\left( \frac{ (y-s)m_{0,x}(s)}{4q_0^3(s)}\right) n_1(s,z)ds\\
			&+\int_{y}^{+\infty}(1+1/z^2)e^{\frac{i}{2}(z-1/z)(y-s)}\frac{ (y-s)m_{0,x}(s)}{4q_0^3(s)}(\partial_sn_1)(s,z)ds\\
			=&\int_{y}^{+\infty}(1+1/z^2)e^{\frac{i}{2}(z-1/z)(y-s)}\partial_s\left( \frac{ (y-s)m_{0,x}(s)}{4q_0^3(s)}\right) n_1(s,z)ds\\
			&+\int_{y}^{+\infty}(1+1/z^2)e^{\frac{i}{2}(z-1/z)(y-s)}\frac{ (y-s)m_{0,x}(s)}{4q_0^3(s)}\left( \frac{-im_0^2}{2 zq_0^2}(n_1+1)+\left( \frac{i m_{0,x}}{2q_0^3}+ \frac{m_0}{2 zq_0^2}\right) n_2 \right) ds.
		\end{align*}
		Hence, we obtain that
		\begin{align*}
			\|k(z)\mathcal{T}_zn\|_{C_b}\leq C(\|m_0\|_{H^{2,1}}),
		\end{align*}
		which together with \eqref{partzn} yeilds
		\begin{align*}
			\|k(z)\partial_zn\|_{C_b}\leq C(\|m_0\|_{H^{2,1}}).
		\end{align*}
		This completes the  proof of Proposition \ref{Propmu}.
	\end{proof}
	For convenience, we denote the banach space $\mathrm{Y}^{(\pm)}$ as a space of $2\times 2$ matrix functions $f(y,z)$ which admits that $f(y,z)\in L^{\infty}(\mathbb{R}^\pm\times\mathbb{R})\cap C_b(\mathbb{R}^\pm, L^2(\mathbb{R}))\cap L^{2}(\mathbb{R}^\pm\times\mathbb{R})$, $\partial_zf(y,z),\ k(z)\partial_zf(y,z)\in C_b(\mathbb{R}^\pm, L^2(\mathbb{R}))$, and the non-diagonal terms of $k(z)(f(y,z)-I)$ belong in $ C_b(\mathbb{R}^\pm, L^{2}(\mathbb{R}))$. The norm of  $\mathrm{Y}^{(\pm)}$ is defined by the sum of the norm listed above.
	
	Note that the operator equations in above proposition are both linear. Thus, $m_0\to\mu^\pm-I$ is locally   Lipschitz continuous from $H^{2,1}(\mathbb{R})$ to the space $\mathrm{Y}^{(\pm)}$ respectively. And when $m_0\equiv0$, it follows that $\mu^\pm\equiv I$. So for any neighborhood $G$ of $0$ in $H^{2,1}(\mathbb{R})$, for any $m_0\in G$ and its corresponding Jost function $\mu^\pm-I$, there exists a constant $C(G)$ such that
	\begin{align*}
		\|\mu^\pm-I\|_{\mathrm{Y}^{(\pm)}}\leq C(G)\|m_0\|_{H^{2,1}}.
	\end{align*}

	Similar to the proof of the above proposition, it follows that as the functions of spectral parameter $z$, $\mu^-_1 (z)$ and $\mu^+_2(z)$ are analytic in $\mathbb{C}^+$ while $\mu^+_1(z)$ and $\mu^-_2(z)$ are analytic in $\mathbb{C}^-$.  In addition, by  Lebesgue's dominated convergence theorem,
	$\lim_{z\to \infty }z(	\mu_\pm(z)-I)$ exist  in their analytic  regions respectively such that
	$\mu_\pm (z)$  admit the following asymptotics in their analytic  regions respectively
	\begin{align}
		\mu_\pm(z)=I+\dfrac{D_1}{z}+\mathcal{O}(z^{-2}),\hspace{0.5cm}z \rightarrow \infty,\label{asymu}
	\end{align}
	where
	\begin{equation}
		D_1(y)=\dfrac{im_{0,x}}{(1+m_0^2)^{3/2}}\sigma_2+\frac{i}{2}\int_{y}^{\pm\infty}\left( \frac{m_{0,x}^2}{q_0^6}+\frac{m_0^2}{q_0^2}\right) ds\sigma_3.
	\end{equation}
	By Abel's formula, $\det\mu^\pm\equiv1$. This asymptotics behavior together with \eqref{symPhi2} also imply that
	\begin{align}
		\mu_\pm(0;y)=F(y)^{-2}=\frac{1}{q_0(y)}I+\frac{im_0(y)}{q_0(y)}\sigma_1.
	\end{align}
	
	To analyze the property near $z=\pm i$,  we let $U(\pm i)$ be  a bounded closed neighborhood of $z=\pm i$ in $\mathbb{C}\setminus\{0\}$ and  define a new transformation
	\begin{equation}
		\mu^{0,\pm} (z)\triangleq\mu^{0,\pm}(z;x) =\Phi_\pm (z)  e^{\frac{k}{2}x\sigma_3},\label{trans2}
	\end{equation}
	which leads to two  Volterra type integrals
	\begin{equation}
		\mu^{0,\pm}(z;x)=I+\int^{x}_{\pm \infty}e^{-\frac{i}{4}(z-\frac{1}{z})(x-s)\hat{\sigma}_3}L_0(s,z)\mu^{0,\pm}(z;s)ds.\label{mu_0}
	\end{equation}
	Then $	\mu^{0,\pm}$ admit the following proposition.
	\begin{Proposition}\label{pro mu0}
		Suppose $m_0\in H^{2,1}(\mathbb{R})$.  Then for $\forall x\in\mathbb{R}$, $	\mu^{0,-}_1(\cdot;x)$ and $ \mu^{0,+}_2(\cdot;x),$ exist uniquely in $L^\infty(U(i))$, and $	\mu^{0,+}_1(\cdot;x)$ and $ \mu^{0,-}_2(\cdot;x),$ exist uniquely in $L^\infty(U(-i))$,  respectively. Moreover, they satisfy the same limits,
		\begin{align}
			&\lim_{z\to i\text{ in } U(i)}\left(\mu^{0,-}_1,\ \mu^{0,+}_2 \right)=\lim_{z\to -i\text{ in } U(-i)}\left(\mu^{0,+}_1,\ \mu^{0,-}_2 \right) =I,\\
			&\lim_{z\to i\text{ in } U(i)}\frac{\left(\mu^{0,-}_1,\ \mu^{0,+}_2 \right)-I}{z-i}=\lim_{z\to -i\text{ in } U(-i)}\frac{\left(\mu^{0,+}_1,\ \mu^{0,-}_2 \right)-I}{z-i}=\mu^{0,(1)},\label{asymu0}
		\end{align}
		where
		$$\mu^{0,(1)}= \left(\begin{array}{cc}
			0 & -\frac{1}{2}(u_0+u_{0,x}) \\
			-\frac{1}{2}(u_0-u_{0,x}) & 0
		\end{array}\right). $$
	\end{Proposition}
	\begin{proof}
		We only present the proof of $\mu^{0,+}_1$. From \eqref{mu_0}, it admits
		\begin{align}
			\mu^{0,+}_1=e_1+\mathcal{T}_0\mu^{0,+}_1,
		\end{align}
		where $\mathcal{T}_0$ is a integral operator on $L^\infty(U(i))$ with integral kernel
		\begin{align*}
			K_0(x,s,z)=\frac{m_0(s)}{4}(z+1/z)\left(\begin{array}{cc}
				0 & 1 \\
				-e^{\frac{i}{2}(z-\frac{1}{z})(x-s)} & 0
			\end{array}\right),\ x<s.
		\end{align*}
		Analogous to the proof of Proposition \ref{Propmu}, $I-\mathcal{T}_0$ is invertible with estimate $\|I-\mathcal{T}_0\|\leq C(\|m_0\|,x,U(i))$. And it immediately comes that $	\lim_{z\to i\text{ in } U(i)}\mu^0=I$. Furthermore,
		\begin{align*}
			\frac{\mu^{0,+}_1-e_1}{z-i}=\frac{z+i}{4z}\int^{x}_{+\infty}m_0(s)\left(\begin{array}{cc}
				0 & 1 \\
				-e^{\frac{i}{2}(z-\frac{1}{z})(x-s)} & 0
			\end{array}\right)\mu^{0,+}_1ds.
		\end{align*}
		Again by Lebesgue's dominated convergence theorem,
		\begin{align*}
			\lim_{z\to i\text{ in } U(i)}\frac{\mu^{0,+}_1-e_1}{z-i}&=\frac{1}{2}\int^{x}_{+\infty}m_0(s)	\lim_{z\to i\text{ in } U(i)}\left(\begin{array}{cc}
				0 & 1 \\
				-e^{\frac{i}{2}(z-\frac{1}{z})(x-s)} & 0
			\end{array}\right)\mu^{0,+}_1ds\\
			&=\frac{1}{2}\int^{x}_{+\infty}m_0(s)	\left(\begin{array}{cc}
				0  \\
				-e^{-(x-s)}
			\end{array}\right)ds=-\frac{1}{2}(u-u_x).
		\end{align*}
		This completes the proof of Proposition \ref{pro mu0}.
			\end{proof}

	The relations (\ref{transmu}) and  (\ref{trans2}) lead  to
	\begin{equation}
		\mu_\pm( z)=F^{-1}(x)\mu^{0,\pm} (z) e^{\frac{i}{4}(z-\frac{1}{z})c_\pm(x)\sigma_3},\label{mu0}
	\end{equation}
	where
	\begin{align}
		c_\pm(x)=\int_{\pm\infty}^x (q_0-1)dy.\label{c+-}
	\end{align}

	\subsection{Reflection coefficient}\label{subec r}
	Since   $\Phi_\pm(z;y)$ are two fundamental matrix solutions of the  Lax  pair (\ref{lax0}),  there exists a linear  relation between $\Phi_+(z;y)$ and $\Phi_-(z;y)$ for $z\in\mathbb{R}$, namely
	\begin{equation}
		\Phi_-(z;y)=\Phi_+(z;y)S(z), \label{scattering}
	\end{equation}
	where $S(z)$ is called scattering matrix
	\begin{equation}
		S(z) =\left(\begin{array}{cc}
			a(z) &-\overline{b(\bar{z})}   \\[4pt]
			b(z) & \overline{a(\bar{z})}
		\end{array}\right),\hspace{0.5cm}\det[S(z)]=1. \nonumber
	\end{equation}
	Combining with  the transformation (\ref{transmu}),  the  equation (\ref{scattering})  is changed to
	\begin{align}
		\mu_-(z)=\mu_+(z)e^{-\frac{i}{4}(z-\frac{1}{z})y\hat{\sigma}_3}S(z). \label{scattering23}
	\end{align}
	 From \eqref{symPhi1}, \eqref{symPhi2} and \eqref{scattering23}, it is shown that  $S(z)$ has the following symmetry reductions
	\begin{equation}
		S(z)=\overline{S(\bar{z}^{-1})}=\sigma_3S\left( -z^{-1}\right) \sigma_3.\label{symS}
	\end{equation}
	Furthermore, by \eqref{scattering23}, $a(z)$ and  $b(z)$ can be expressed by $\mu_\pm (z) $ at $x=0$ as
	\begin{align}
		&a(z)=\mu_-^{11}(z;0) \overline{\mu_+^{11} (z;0)}+\mu_-^{21} (z;0)\overline{\mu_+^{21} (z;0)},\label{scatteringcoefficient1}\\ &\overline{b(z)}=\mu_-^{11}(z;0)\mu_+^{21}(z;0)-\mu_-^{21}(z;0)\mu_+^{11}(z;0).\label{scatteringcoefficient2}
	\end{align}
	So $a(z)$ is analytic on $\mathbb{C}^+$ and continuous on $\mathbb{R}$. From  \eqref{asymu} and \eqref{scatteringcoefficient1},    we obtain  the asymptotic   of $a(z)$
	\begin{align}
		a(z)=1+\mathcal{O}(z^{-1}),\hspace{0.5cm}	b(z)=\mathcal{O}(z^{-1}),\hspace{0.5cm}z \rightarrow \infty.\label{asya}
	\end{align}
	On the other hand,  taking $z \rightarrow i$ in \eqref{scatteringcoefficient1} and  combining the expansions in \ref{mu0} and \ref{asymu0},   we get  the asymptotic of $a(z)$,
	\begin{align}
		a(z)=e^{\frac{1}{2}\int_{\mathbb{R}}(q_0-1)dx}\left(1+ \mathcal{O}\left( (z-i)^2\right)\right), \hspace{0.3cm}\text{as }z\to i.\label{asyMi}
	\end{align}
	
	The function $a(z)$ may have zeros on $\mathbb{C}^+\cup \mathbb{R}$, which is equivalent to an eigenvalue or a resonance of the spectral problem \eqref{lax0}.
	%A point  $z_0\in\mathbb{C}^+$ satisfies  $a(z_0)=0$ if and only if  $z_0$ is of the spatial part of the Lax pair  \eqref{lax y}.  An $z_0\in\mathbb{R}$ is called a   if $a(z_0)=0$.
	%the linear equation \eqref{lax0} with $z_0$ admits a column vector solution in $L^\infty(\mathbb{R})$ with the asymptotic behavior\begin{align*}	\Phi(x,z_0)\sim \left\{ \begin{array}{ll}		 c_+e^{-\frac{i}{4}(z-\frac{1}{z})\left(x-\int_{x}^{+\infty} (\sqrt{m(s)^2+1}-1) ds\right)}e_1,   &\text{as } x\to+\infty,\\[12pt]		c_-e^{\frac{i}{4}(z-\frac{1}{z})\left(x-\int_{x}^{+\infty} (\sqrt{m(s)^2+1}-1) ds\right)}e_2  , &\text{as }x\to-\infty.\\	\end{array}\right.\end{align*}where $c_+$ and $c_-$ are nonzero constant coefficients. The  resonance  is eliminated in our paper. The brings some deficiency to estimate the Riemann-Hilbert problem, so the research of resonance need another way to estimate, for example, as in   \cite{Zhou1989CPAM,Jenkins}.
	From \eqref{intmu} and \eqref{scattering23},  it follows  that
	\begin{align}
		a(z)=1+\int_{\mathbb{R}}\left(P\mu_- \right)_{11} dx,
	\end{align}
	with $P$ defined in \eqref{F}.
	Thus $\|a-1\|_\infty$ is controlled by the $L^\infty$ norm of Jost function, so from Proposition \ref{Propmu}, $\|a-1\|_\infty$ is controlled by $\|m_0\|_{H^{1,1}}$. When  $\|m_0\|_{H^{1,1}}$ is small, $a$ must has no zero on $\mathbb{C}^+\cup \mathbb{R}$.
	
	We define  the \textit{reflection coefficients}   by
	\begin{equation}
		r(z)=\frac{b(z)}{a(z)},\hspace{0.5cm}	\tilde{r}(z)=\frac{b(z)}{a^*(z)},\hspace{0.5cm} z\in\mathbb{R}.\label{symr}
	\end{equation}
		From the symmetry of $a$ and $b$ in \eqref{symS},
	$r(z)$ and $\tilde{r}(z)$ admit  symmetry reductions
	\begin{equation}
		r(z)=\overline{r(z^{-1})}=r(-z^{-1}),\quad 	\tilde{r}(z)=\overline{\tilde{r}(z^{-1})}=\tilde{r}(-z^{-1}).\nonumber
	\end{equation}
	In addition, the asymptotic behavior  of $a,$ $b$ in \eqref{asya} shows that $r,\ \tilde{r}=\mathcal{O}(z^{-1}),\ z \rightarrow \infty$. Then the symmetry in \eqref{symr}  leads to   $r(0)=\tilde{r}(0)= 0$.
	The above symmetries of $r$ and $\tilde{r}$ also imply that $r(z(k))$ and $\tilde{r}(z(k))$ are  well-defined as the functions of $k$ with $k=z-1/z$.
	%Denote \begin{align} 	\rho(k)=r(z(k)).\label{def rho} \end{align}
	%  For convenience, when  $\rho\in H^{1,1}(\mathbb{R})$, we also call that $r \in H^{1,1}(\mathbb{R}_k)$ with  \begin{align*} 	\|r\|_{H^{1,1}(\mathbb{R}_k)}=\|\rho\|_{H^{1,1}}.  \end{align*}
	\begin{Proposition}\label{lemmar1}
		If  initial data  $m_0 \in   H^{2,1}(\mathbb{R})$ such that the Lax pair  \eqref{lax0} has no any resonance and eigenvalue,   then the reflection coefficients $r,\  \tilde{r}\in H^{1,1}(\mathbb{R}_k)$.
	\end{Proposition}
	\begin{proof}
		The proof is given by taking $r$ as an example, since $r$ and $\tilde{r}$ have same property.
		Noting that for any function $f(z)$ satisfying $f(z)=f(-1/z)$, by denoting $z_+(k)=(k+\sqrt{k^2+4})/2$, $z_-(k)=(k-\sqrt{k^2+4})/2$, it is shown that
		\begin{align*}
			\int_\mathbb{R}|f(z)|^pdz=	\int_\mathbb{R}\frac{|f(k)|^p}{1+z_-(k)^{-2}}dk+\int_\mathbb{R}\frac{|f(k)|^p}{1+z_+(k)^{-2}}dk=\int_\mathbb{R}|f(k)|^pdk.
		\end{align*}
		Together with $|\partial_kf|= |\partial_zf||\partial_k/\partial_z|\leq C|\partial_zf|$,  it is sufficient to prove that $k(\cdot )r$, $k(\cdot )r' \in L^2(\mathbb{R})$.
		For $m_0\in \mathsf{U}$,  $1/a$ is bounded on $\mathbb{R}$. And Proposition \ref{Propmu} gives that $\|(\cdot)b\|_2\leq \|\mu_-^{11}\|{L^\infty}\|(\cdot)\mu_+^{21}\|_{C}+\|(\cdot)\mu_-^{21}\|_{C}\|\mu_+^{11}\|{L^\infty}$, which leads to $zr(z)\in L^2(\mathbb{R})$. And for the $z$-derivative of $r$, \eqref{symr} also gives that
		\begin{equation*}
			r'(z)=\frac{a(z)b'(z)-a'(z)b(z)}{a(z)^2},
		\end{equation*}
		where $a$, $b\in L^\infty(\mathbb{R})$ and $(\cdot)a',\ (\cdot)b'\in L^2(\mathbb{R})$ via Proposition \ref{Propmu}. And combining $(\cdot)b,\ (\cdot)b'\in L^2(\mathbb{R})$, it follows that $(\cdot)b\in L^\infty(\mathbb{R})$. Then from $(\cdot)^2b'\in L^2(\mathbb{R})$, we arrive at that $(\cdot)^jr'\in L^2(\mathbb{R})$ for $j=0,1$. Finally from $r(z)=\overline{r(\bar{z}^{-1})}$, it is adduced that $\langle k\rangle r$, $\langle k\rangle r'\in L^2(\mathbb{R})$. Thus we obtain the desired result of Proposition \ref{lemmar1}.
	\end{proof}
	The above proposition also shows that the maps $m_0\to a-1$, $m_0\to a'$ and $m_0\to b$ are locally Lipschitz continuous from $ H^{2,1}(\mathbb{R} ) $ to $L^\infty(\mathbb{R} )$, $L^{2,1}(\mathbb{R} )$ and $H^{1,1} (\mathbb{R} )$ respectively.
	Therefore, we finally arrive at the following lemma.
	\begin{lemma}\label{lemmasec2}
	Assume that  $m_0 \in   H^{2,1}(\mathbb{R})$ such that the Lax pair  \eqref{lax0} has no any resonance and eigenvalue. Then there exist two reflection coefficients  $r$ and $ \tilde{r}$ determined by the $m_0$, and the two  maps:
	\begin{align*}
		&m_0 \in  H^{2,1}(\mathbb{R})\longrightarrow r\in H^{1,1}(\mathbb{R}_k),\text{ and}\\
		&m_0 \in    H^{2,1}(\mathbb{R})\longrightarrow \tilde{r}\in H^{1,1}(\mathbb{R}_k),
	\end{align*}
	are locally Lipschitz continuous.
\end{lemma}
Note that when $m_0\equiv0$, then it follows that $\tilde{r}=r\equiv 0$. Thus the following corollary is obtained from Lemma \ref{lemmasec2} directly.
\begin{corollary}\label{rk}
	If  initial data  $m_0\in  H^{2,1}(\mathbb{R} ) $ admits
	no eigenvalues or spectral singularity and $\|m_0\|_{H^{2,1}}\leq \rho_0$ for some $\rho_0>0$,   then the reflection coefficients $r,\ \tilde{r}\in H^{1,1}(\mathbb{R}_k)$ with
	\begin{align}
		\|	\tilde{r}\|_{H^{1,1}(\mathbb{R}_k)},\ \|r\|_{H^{1,1}(\mathbb{R}_k)}\leq  C(\rho_0) \|m_0\|_{H^{2,1}}.
	\end{align}
\end{corollary}

	\section{Inverse scattering transform }\label{sec3}
	In this section, we shall develop an analytic framework for exploring complex analysis of the RH problems and determine their solvability based on a given reflection coefficient.

		\subsection{The RH problems to  the Cauchy problem  }\label{sec2.3}
	For convenience, denote the  phase function
	\begin{equation}
		\theta(z;y)= -\frac{i}{4}\left(z-\frac{1}{z} \right) y.\label{theta}
	\end{equation}
	By the Jost functions $\mu^\pm(z;y)$ and the function $a(z)$,  two  piecewise  analytical  matrices are defined as follows,
	\begin{align}
	 M_l(z;y)=\left\{ \begin{array}{ll}
			\left( \frac{ \mu^-_1  (z;y) } {a(z)},  \mu^+_2  (z;y)\right),   &\text{as } z\in \mathbb{C}^+,\\[12pt]
			\left(  \mu^+_1(z;y),\frac{ \mu^-_2(z;y)}{a^*(z)}\right)  , &\text{as }z\in \mathbb{C}^-,\\
		\end{array}\right.\label{ml}\\
	 M_r(z;y)=\left\{ \begin{array}{ll}
			\left( \mu^-_1  (z;y), \frac{\mu^+_2  (z;y)  } {a(z)} \right),   &\text{as } z\in \mathbb{C}^+,\\[12pt]
			\left( \frac{\mu^+_1(z;y) }{a^*(z)} ,\mu^-_2(z;y)\right)  , &\text{as }z\in \mathbb{C}^-.\\
		\end{array}\right.\label{mr}
	\end{align}
	which solve the following RH problems, respectively.
	\begin{RHP}\label{RHP1}
		\hfill
		\begin{itemize}
			\item[(1)] Analyticity: $M_l(z)$ is meromorphic in $\mathbb{C}\setminus \mathbb{R}$;
			
			\item [(2)]Symmetry:
			\begin{align}
				M_l(z)=\sigma_3\overline{M_l(-\bar{z})}\sigma_3=\sigma_2\overline{M_l(\bar{z})}\sigma_2=\sigma_3M_l(0)^{-1}M_l(-1/z)\sigma_3;\label{syM}
			\end{align}
			
			\item [(3)]Jump condition: $M_l$ has continuous boundary values $[M_l]_\pm(z)$ on $\mathbb{R}$ and
			\begin{equation*}
				[M_l]_+(z)=[M_l]_-(z)\tilde{V}(z),\hspace{0.5cm}z \in \mathbb{R},
			\end{equation*}
			where
			\begin{equation}
				V_l(z)=\left(\begin{array}{cc}
					1+|r(z)|^2 & e^{2\theta(z)}\overline{r(z)}\\
					e^{-2\theta(z)}r(z) & 1
				\end{array}\right);\label{jumpv}
			\end{equation}
			
			\item [(4)]Asymptotic behavior: $	M_l(z) = I+\mathcal{O}(z^{-1}),\hspace{0.5cm}z \rightarrow \infty.$
		\end{itemize}
	\end{RHP}
	
	\begin{RHP}\label{RHP2}
		\hfill
		\begin{itemize}
			\item[(1)]  Same Analyticity, Symmetry and Asymptotic behavior  as in RHP \ref{RHP1};
			
			\item[(2)] Jump condition: $M_r(z)$ has continuous boundary values $[M_r]_\pm(z)$ on $\mathbb{R}$ and
			\begin{equation}
				[M_r]_+(z)=[M_r]_-(z)V(z),\hspace{0.5cm}z \in \mathbb{R},
			\end{equation}
			where
			\begin{equation}
				V_r(z)=\left(\begin{array}{cc}
					1 & e^{2\theta(z)}\overline{\tilde{r}(z)}\\
					e^{-2\theta(z)}\tilde{r}(z) & 1+|\tilde{r}(z)|^2
				\end{array}\right).
			\end{equation}
		\end{itemize}
	\end{RHP}

Using $M_r(z)$ and $M_l(z)$ instead of $M_r(z;y)$ and $M_l(z;y)$ underscores the fact that $M_r$ and $M_l$ are matrix functions dependent on $z$, with $y$ serving as a parameter. Consequently, in accordance with Liouville's theorem, it follows that there is a unique solution for each of these two RH problems. This in turn implies that $M_r$ and $M_l$ respectively represent the unique solutions to these two RH problems.

%	Here,  the expression
%$M_r(z)$ and $M_l(z)$ instead of $M_r(z;y)$ and $M_l(z;y)$ is used to  emphasize that $M_r$, $M_l$ are the matrix functions of $z$, and $y$ is only a parameter. It then %follows from Liouville's theorem that there is only solution for  each of these two RH problems. This then implies that $M_r$ and  $M_l$ are the unique solution of these %two RH problems %respectively.

	For convenience, we define
	\begin{align}
		M(z;y)=\left\{ \begin{array}{ll}
			M_l(z;y),   &\text{as } y\in \mathbb{R}^+,\\[12pt]
			M_r(z;y) , &\text{as }y\in \mathbb{R}^-,\\
		\end{array}\right.\label{M1}
	\end{align}
	Thus, 	from the asymptotic behaviors of the functions $\mu_\pm$ and \eqref{asyMi}, we arrive at  reconstruction formula as follows.
	\begin{lemma}\label{lemm3.2}
		The reconstruction formula
		is given by
		\begin{equation}
			q(y)=\frac{1}{M_{11}(0)},\label{recons q}
		\end{equation}
		where
		\begin{equation}
			x(y)=y+c_+(x)=y-\ln\left( \frac{M_{12}(i)+M_{22}(i)}{M_{11}(i)+M_{21}(i)}\right).\label{recons x}
		\end{equation}
		And as $z\to\infty$, it is shown from \eqref{asymu} that
		\begin{align}
			&\lim_{z\to \infty}z\left(M_l-I\right)=i\eta\sigma_2+\zeta^{(+)}\sigma_3,\ \lim_{z\to \infty}z\left(M_r-I\right)=\eta\sigma_2+\zeta^{(-)}\sigma_3\label{recons m},
		\end{align}
		where
		\begin{align}
			&\eta(y)=\dfrac{m_{0,x}(y)}{(1+m_0(y)^2)^{3/2}},\ \zeta^{(\pm)}(y)=\frac{i}{2}\int_{y}^{\pm\infty}\left( \frac{m_{0,x}(s)^2}{q(s)^6}+\frac{m_0(s)^2}{q(s)^2}\right) ds.
		\end{align}
	\end{lemma}
	On the other hand, the following proposition reveals that  the solutions reconstructed from two RH problems in above formulas  are actually same at $ y = 0. $
	\begin{Proposition}\label{Proeq}
		For $y\in\mathbb{R}^+$, denote $q_l$, $\eta_l$ and $\zeta_l$ as those  recovered  from   $M_l(z;y)$, and for  $y\in\mathbb{R}^-$  denote $q_r$, $\eta_r$ and $\zeta_r$  as those  recovered  from   $M_r(z;y)$. Then  $\eta_l(0)=\eta_r(0)$, $\zeta_l(0)=\zeta_r(0)$ and  $q_l(0)=q_r(0)$.
	\end{Proposition}
	\begin{proof}
		Denote
		\begin{align}
			\tilde{a}(z)=\left\{ \begin{array}{ll}
				a(z),   &\text{as } z\in \mathbb{C}^+,\\[12pt]
				a^*(z)^{-1} , &\text{as }z\in \mathbb{C}^-,\\
			\end{array}\right. \hspace{0.3cm}	\tilde{a}\to I,\ |z|\to\infty.
		\end{align}
		Then it immediately follows that $M_l(z;y)=M_r(z;y)	\tilde{a}(z)^{-\sigma_3}$. Thus the result  is obtained from  \eqref{asya}, \eqref{asyMi},  $a(0)=1$ and Lemma \ref{lemm3.2}.
	\end{proof}
	
	\subsection{Solvability of the RH problems}
	In this subsection, firstly, we  prove  solvability of  the RH problem \ref{RHP2} about $M_l$ for $y\in\mathbb{R}^+$  under
	the given reflection coefficient $r\in H^{1,1}(\mathbb{R}_k)$.

 We first give the following definition.
	For  any function $f(z) \in L^p(\mathbb{R})$, $1\leq p<\infty$, the Cauchy operator $\mathcal{C}$ is defined  by
	\begin{align}
		\mathcal{C}f(z)=\frac{1}{2\pi i}\int_{\mathbb{R}}\frac{f(s)}{s-z}ds, \ z\in\mathbb{C}\setminus\mathbb{R}.
	\end{align}
	The function $\mathcal{C}f$ is analytic off the real line such that $\mathcal{C}f(\cdot+i\gamma)$ is in $L^p(\mathbb{R})$ for each $0\neq\gamma\in\mathbb{R}$.
	When $z$ approaches to a point on the real line transversely from the upper and lower
	half planes, that is, if $\gamma\to \pm0$ in the following expression, the Cauchy operator $\mathcal{C}$ becomes the  Cauchy projection operator $\mathcal{C}_\pm$:
	\begin{align}
		&\mathcal{C}_\pm f(z)=\lim_{\gamma\to 0}\frac{1}{2\pi i}\int_{\mathbb{R}}\frac{f(s)}{s-(z\pm i\gamma )}ds, \ z\in\mathbb{R}
	\end{align}
	which admits
	\begin{align}
		&(\mathcal{C}_\pm f)^\wedge (z)=\pm X_\pm(z)	 \hat{f}(z),
	\end{align}
	where  $ X_\pm$ denotes the characteristic function on $\mathbb{R}^\pm$.
	
	For any $ f\in L^p(\mathbb{R})$, $1\leq p<\infty$, the Cauchy integral  $\mathcal{C} f$ is analytical  off
	the real line, decays to zero as $|z|\to\infty$, and approaches to $\mathcal{C}_\pm f $ almost everywhere
	when a point $z \in\mathbb{C}^\pm$ approaches to a point on the real axis by any non-tangential contour
	from $\mathbb{C}^\pm$ \cite{book}. If $1< p<\infty$, then there exists a positive constant $C_p$   such that
	\begin{align}
		\parallel \mathcal{C}_\pm f\parallel_{L^p}\leq C_p\parallel f \parallel_{L^p}.
	\end{align}
	Moreover, $C_2=1$. When $f\in L^1(\mathbb{R})$, as $z\to\infty$, $\mathcal{C} f(z)=\mathcal{O}(z^{-1})$.

	%For convenience, $M^{(1)}_l$ is abbreviated as $M^{(1)}$ in this section from the definition of $M^{(1)}$ in \eqref{M1}.
	For $z\in\mathbb{R}$, denote $w=w_++w_-$
	with
	\begin{align}
		&   w_+(z;y)=\left(\begin{array}{cc}
			0  & 0\\
			r(z)e^{-2i\theta(z;y)}	& 0
		\end{array}\right),\ w_-(z;y)=\left(\begin{array}{cc}
			0  & \bar{r}(z)e^{2i\theta(z;y)}\\
			0	& 0
		\end{array}\right).
	\end{align}
	In view of  the Beals-Coifman  theorem in \cite{BC1984}, the  solution of RH problem \ref{RHP2} is given by
	\begin{align}
		M(z;y)=&I+\mathcal{C}(\mu w)(z;y), \label{intM}\ z\in\mathbb{C}
	\end{align}
	and it exists if and only if there exists
	  a solution $\mu(z;y)$ of the Fredholm integral equations
	\begin{align}
		\mu(z;y)=&I+\mathcal{C}_w(I)(z;y)+\mathcal{C}_w(	\mu-I)(z;y),\label{mu}
	\end{align}
	with $$\mathcal{C}_wf(z;y)=\mathcal{C}_+(fw_-)(z;y)+\mathcal{C}_-(fw_+)(z;y).$$ And  a simple calculation gives that $\|\mathcal{C}_wI\|_{2}\leq C(\|r\|_2). $
	We rewrite \eqref{mu} as
	\begin{align}
		\mu(z;y)-I=\mathcal{K}(I)(z;y)+\mathcal{K}(\mu-I)(z;y),\ z\in\mathbb{R},
	\end{align}
	where $\mathcal{K}$ is a operator on $L^2(\mathbb{R})$ such that for any matrix function $f=(f_1,f_2)\in L^2(\mathbb{R})$,
	\begin{align}
		&\mathcal{K}(f)=(\mathcal{A}(f_2), \mathcal{B}(f_1)),\\
		&\mathcal{A}(f_2)=\mathcal{C}_-(re^{-2i\theta}f_2),\hspace{0.5cm}\mathcal{B}(f_1)(z;y)=\mathcal{C}_+(\bar{r}e^{2i\theta}f_1).
	\end{align}
	$\mathcal{K}$  can also be  a operator on the row function in $L^2(\mathbb{R})$.
	The following proposition is given to show the invertibility of $I-\mathcal{K}$.
	\begin{Proposition}\label{lemmaG}
		Assume that $r\in H^s(\mathbb{R})$, $s>1/2$.  Then   $I-\mathcal{K}$ is a bounded Fredholm operator: $ L^2(\mathbb{R})\to  L^2(\mathbb{R})$ for $y\in\mathbb{R}^+$ with index zero. Moreover, $(I-\mathcal{K})^{-1}$ exists and is also a bounded linear operator: $ L^2(\mathbb{R})\to  L^2(\mathbb{R})$. And     there exists a constant $C$ such that
		\begin{align}
			\parallel\left( I-\mathcal{K}\right) ^{-1}\parallel\leq C.
		\end{align}
	\end{Proposition}
	\begin{proof}
		We choose the first row of \eqref{mu} for example to give the proof and denote it as $(\mu_{11}-1,\mu_{12})$. Recall  the definition of $\theta$ in \eqref{theta}. Then
		\begin{align}
			|e^{-2i\theta(z_j)}|=e^{-\frac{y}{2}(1+1/|z_j|^2)\text{Im}z_j}=|e^{2i\theta(\bar{z}_j)}|.
		\end{align}
		For any $f=(f_1,f_2)\in L^2(\mathbb{R})$,
		\begin{align*}
			\mathcal{K}^2f=\left(\mathcal{A}\mathcal{B}f_1, \mathcal{B}\mathcal{A}f_2 \right) .
		\end{align*}
		Taking $\mathcal{A}\mathcal{B}$ as an example, it is found that
		\begin{align*}
			(\mathcal{A} \mathcal{B} (f_1^\vee))^\wedge (z)={\mathcal{C}_-(\mathcal{C}_+(f^\vee w_-)w_+)}^\wedge (z) =\int_{\mathbb{R}}K(z,\lambda;y)
			f_1(\lambda)d\lambda,
		\end{align*}
		where $K(z,\lambda;y)$ is a scalar function with
		\begin{align*}
			K(z,\lambda;y)=-X_-(z)\int_{\mathbb{R}}X_+(s)	{(\bar{r}e^{-2\theta(\cdot;y)})}^\wedge(z-s)(re^{2\theta(\cdot;y)})^\wedge(s-\lambda)ds.
		\end{align*}
		Since   $$\|K(\cdot,\cdot; y)\|_{L^2(\mathbb{R}\otimes\mathbb{R})}\leq \|(\bar{r}e^{-2\theta(\cdot;y)})^\wedge\|_{{2,s}}\|(re^{2\theta(\cdot;y)})^\wedge\|_{{2,s}}\leq C(y)\|r\|_{H^s}^2$$ for a positive constant $C(y)$ which is finite for every $y\in\mathbb{R}^+$, it follows that $\mathcal{A} \mathcal{B} =\mathcal{C}_-(\mathcal{C}_+(\cdot w_-)w_+)$ is the Hilbert-Schmidt operator. It in turn implies that $\mathcal{K}^2$ is compact. Therefore, %it is inferred in  \cite{simon} that
		$I-\mathcal{K}^2$ is a Fredholm operator with index zero, and so is $I-\mathcal{K}$. Therefore, to prove invertibility of $I-\mathcal{K}$, it is sufficient to prove that $I-\mathcal{K}$ is injection. If there exists a matrix function $f\in L^2(\mathbb{R})$ such that $(I-\mathcal{K})f=0$, we denote the first row of  $f$ as $(f_{11},f_{12})$. Define
		\begin{align}
			g_{11}=\mathcal{C}(re^{-2\theta}f_{12}),\ g_{12}=\mathcal{C}(\bar{r}e^{2\theta}f_{11}).
		\end{align}
		Thus, it appears that for $z\in\mathbb{R}$,
		\begin{align*}
			&[g_{11}]_-(z)=f_{11}(z), \hspace{2.4cm}[g_{11}]_+(z)=f_{11}(z)+re^{-2\theta}f_{12}(z),\\
			&[g_{12}]_-(z)=f_{12}(z)+\bar{r}e^{2\theta}f_{11}(z),\ \ [g_{12}]_+(z)=f_{12}(z).
		\end{align*}
		And the function $g_{11}g_{11}^*+g_{12}g_{12}^*$ is a analytic function on $\mathbb{C}^+$. Using the result in the Cauchy-Goursat theorem gives that
		\begin{align*}
			0=\int_{\mathbb{R}_+}g_{11}(z)g_{11}^*(z)+g_{12}(z)g_{12}^*(z)dz,
		\end{align*}
		where $\mathbb{R}_+$ is used to denote $\mathbb{R}$ when it is viewed as the boundary of $\mathbb{C}^+$.
		On the other hand
		\begin{align*}
			0=\int_{\mathbb{R}^+}g_{11}(z)g_{11}^*(z)+g_{12}(z)g_{12}^*(z)dz&=\int_{\mathbb{R}}|f_{11}(z)|^2+|f_{12}(z)|^2dz,
		\end{align*}
		which implies that   $\|f_{11}\|_2=\|f_{12}\|_2=0$. And the second row of $f$ can obtain the same result, which finally implies that $f=0$, $a.e.$ on $\mathbb{R}$.
		
		The last step is to prove that  $\|(I-\mathcal{K})^{-1}\|$ is uniformly bounded for $y\in\mathbb{R}^+$.  Similarly, we only give the estimate of the  row function.
		For any matrix function $h=(h_{11},h_{12})\in L^2(\mathbb{R})$, denote
		\begin{align*}
			&f^{(+)}=(f^{(+)}_1,f^{(+)}_2)=(I-\mathcal{K})^{-1}(\mathcal{C}_-h_{11},\mathcal{C}_+h_{12}),\\
			&f^{(-)}=(f^{(-)}_1,f^{(-)}_2)=(I-\mathcal{K})^{-1}(\mathcal{C}_+h_{11},\mathcal{C}_-h_{12}).
		\end{align*}
		Then it holds that $f^{(+)}-f^{(-)}=(I-\mathcal{K})^{-1}(-h_{11},h_{12})$.
		And denote the analytic functions on $\mathbb{C}\setminus\mathbb{R}$ with
		\begin{align}
			g_{11}^{(\pm)}=\mathcal{C}(re^{-2\theta}f^{(\pm)}_{2}),\hspace{0.5cm}g_{12}^{(\pm)}=\mathcal{C}(\bar{r}e^{2\theta}f_{1}^{(\pm)}).
		\end{align}
		For $f^{(+)}$, the difference is that in this case, for $z\in\mathbb{R}$,
		\begin{align*}
			&[g^{({+})}_{11}]_-(z)=f^{(+)}_{1}(z)-\mathcal{C}_-h_{11}(z),\\ &[g^{(+)}_{11}]_+(z)=f^{(+)}_{1}(z)-\mathcal{C}_-h_{11}(z)+re^{-2\theta}f^{(+)}_{2}(z),\\
			&[g^{(+)}_{12}]_-(z)=f^{(+)}_{2}(z)-\mathcal{C}_+h_{12}(z)-\bar{r}e^{2\theta}f^{(+)}_{1}(z),\\ &[g^{(+)}_{12}]_+(z)=f^{(+)}_{2}(z)-\mathcal{C}_+h_{12}(z).
		\end{align*}
		Performing the same manipulations as above, we obtain that
		\begin{align*}
			0&=\int_{\mathbb{R}_+}g^{(+)}_{11}(z)(g^{(+)}_{11}+\mathcal{C}h_{11})^*(z)+(g^{(+)}_{12}+\mathcal{C}h_{12})(z)(g^{(+)}_{12})^*(z)dz\\
			&=\int_{\mathbb{R}}|f^{(+)}_{1}(z)|^2+|f^{(+)}_{2}(z)|^2-(f^{(+)}_{2}(z)\overline{\mathcal{C}_+h_{2}(z)}+\mathcal{C}_-h_{1}(z)\overline{f^{(+)}_{1}(z)})dz,
		\end{align*}
		which leads to that
		\begin{align*}
			\|f^{(+)}\|_2^2&=\int_{\mathbb{R}}|f^{(+)}_{1}(z)|^2+|f^{(+)}_{2}(z)|^2dz=\big|\int_{\mathbb{R}}f^{(+)}_{2}(z)\overline{\mathcal{C}_+h_{2}(z)}+\mathcal{C}_-h_{1}(z)\overline{f^{(+)}_{1}(z)}dz\big|\leq \|f^{(+)}\|_2\|h\|_2.
		\end{align*}
		Therefore, it is concluded  that $	\|f^{(+)}\|_2\leq \|h\|_2$. For $f^{(-)}$, note that for $z\in\mathbb{R}$,
		\begin{align*}
			[g^{({-})}_{11}]_-(z)=f^{(-)}_{1}(z)-\mathcal{C}_+h_{11}(z),
		\end{align*}
		and for $z\in\mathbb{C}\setminus\mathbb{R}$,
		\begin{align*}
			\mathcal{C}([g^{({-})}_{11}]_-)(z)=\left\{ \begin{array}{ll}
				0,   & z\in \mathbb{C}^+,\\[12pt]
				g^{({-})}_{11} , &z\in \mathbb{C}^-,\\
			\end{array}\right.=\mathcal{C}(f^{(-)}_{1}-\mathcal{C}_+h_{11})(z),
		\end{align*}
		which implies that $\mathcal{C}_-(f^{(-)}_{1})(z)\equiv0$ and $f^{(-)}_{1}(z)=\mathcal{C}_+h_{11}(z)$. Thus, $g^{({-})}_{11}\equiv0$. Similarly, it also follows that  $f^{(-)}_{2}(z)=\mathcal{C}_-h_{12}(z)$. Then we arrive at that $\|f^{(-)}\|_2\leq \|h\|_2$. Therefore, it  transpires that
		\begin{align*}
			\|	f^{(+)}-f^{(-)}\|_2=\|(I-\mathcal{K})^{-1}(-h_{11},h_{12})\|_2\leq 2 \|h\|_2,
		\end{align*}
		which in turn implies  $\|(I-\mathcal{K})^{-1}\|\leq 2 $. Thus we arrive at the desired result of Proposition \ref{lemmaG}.
	\end{proof}
	As a consequence of this proposition, the solution of \eqref{mu} exists with
	\begin{align}
		\mu(z;y)=I+(I-\mathcal{K})^{-1}(\mathcal{K}(I))(z;y),
	\end{align}
	and
	\begin{align}
		\|\mu(z;y)-I\|_2\leq 2 \|\mathcal{K}(I)\|_2\leq 2\|r\|_2.\label{muL2}
	\end{align}
	Therefore, it follows from \eqref{intM} that
	\begin{align}
		\|M(z;y)-I\|_2=\|\mathcal{C}( w)(z;y)+\mathcal{C}((\mu-I) w)(z;y)\|\leq (2\|r\|_\infty+1)\|r\|_2,
	\end{align}
	When $y\in\mathbb{R}^-$, following the same process as $M_l$, $M_r$ admits similar property shown in  proposition below.
	\begin{Proposition}\label{Pro3.2}
		If    $  \tilde{r}\in  H^{s}(\mathbb{R})$, $s>1/2$, $y\in\mathbb{R}^-$, then  there  exists a unique $M_r$  with the estimate
		\begin{align}
			&\parallel M_r-I\parallel_{2}\leq \|\tilde{r}\|_2(2\|\tilde{r}\|_\infty+1).
		\end{align}
	\end{Proposition}
	%\begin{lemma}\label{lemmaextence}
	%	For any scattering data in $H^{s}(\mathbb{R}_k)$, the solution of the RH problems in RHP \ref{RHP1} and \ref{RHP2} exist unique respectively.
	%\end{lemma}
	\begin{remark}
		When eigenvalues exist and are simple,  there  also exists a unique solution for these two RH problems, respectively.
		Denote the solution of RHP \ref{RHP1} with  eigenvalues by $M$. It is inferred that $|M-I|$ is controlled by $\|r\|_{H^{s}}$ and the absolute value of eigenvalues.
		But unlike the case in \cite{YFL}, the location of eigenvalue in our case cannot be controlled by $\|m_0\|_{H^{1,1}}$. %So it is not easy to control $|M-I|$  via initial value, which will bring difficulty in the estimates of the solution latter.
		Then it is not easy to control $|M-I|$  via initial value $m_0$, which will bring difficulty in the estimates of the solution of \eqref{mch}.
		%To avoid this difficulty in the estimates of the solution, we use small norm assumption to  the existence of eigenvalue.
		It is necessary to use small norm assumption to exclude the existence of eigenvalue.
	\end{remark}
	\subsection{Estimates on the solution of RH problems}
	The aim of this subsection is to give some basic estimates on the solution of RH problems.
	Consider the $y$-derivative at both sides of the equation \eqref{mu}
	\begin{align}
		\partial_y\mu(z;y)=(\partial_y\mathcal{K})(I)(z;y)+(\partial_y\mathcal{K})(	\mu-I)(z;y)+\mathcal{K}	(\partial_y\mu)(z;y),\label{muy}
	\end{align}
	where for any $f\in L^2$
	\begin{align*}
		(\partial_y\mathcal{K})f(z;y)=&\mathcal{C}_+(f\partial_yw_-)(z;y)+\mathcal{C}_-(f\partial_yw_+)(z;y)\\
		=&\frac{1}{2}\mathcal{C}_+(ik(\cdot)fw_-)(z;y)+\frac{1}{2}\mathcal{C}_-(ik(\cdot)f\partial_yw_+)(z;y).
	\end{align*}
	It holds that
	\begin{align*}
		\|(\partial_y\mathcal{K})(I)\|_{2}\leq \|k(\cdot)r\|_{{2}},\hspace{0.5cm}\|(\partial_y\mathcal{K})(	\mu-I)\|\leq\|\mu-I\|_2\|\|k(\cdot)r\|_{{\infty}}.
	\end{align*}
	It is thereby inferred that $	\partial_y\mu$ exists with
	$\partial_y\mu=(I-\mathcal{K})^{-1}((\partial_y\mathcal{K})(I)+(\partial_y\mathcal{K})(	\mu-I))$, and
	\begin{align}
		\|	\partial_y\mu\|_2\leq& \|\mathcal{C}_+(k(\cdot)\bar{r}e^{2\theta})\|_2+\|\mathcal{C}_-(k(\cdot)re^{-2\theta})\|_2\nonumber\\
		&+4(\|\mathcal{C}_+(\bar{r}e^{2\theta})\|_2+\|\mathcal{C}_-(re^{-2\theta})\|_2)\|k(\cdot)r\|_{{\infty}}\label{muyL2}
	\end{align}
	Therefore, we arrive at the $y$-derivative  of $M(z;y)$  with
	\begin{align}
		&\partial_y	M(z;y)=\mathcal{C}((\partial_y\mu) w)(z;y)+\frac{i}{2}\mathcal{C}(k(\cdot)\mu w)(z;y).
	\end{align}
	
	Next, we provide  estimates of $M$ and its $y$-derivative  at $z= i$ and $z=0$, which will be used  latter. Note that $M$ has no jump at $z=0$. Then $M(0)$ is well-defined.  $M$ is analytic at $z=i$ with expansion
	\begin{align}
		M(z)=M(i)+M^{i,(1)}(z-i)+\mathcal{O}((z-i)^2).
	\end{align}
	The following proposition shows that $M(i)$ and $M(0)$ are uniformly bounded for $y\in\mathbb{R}^+$.
	\begin{Proposition}\label{proMi}
		If $r\in H^{1,1}(\mathbb{R}_k)$ satisfies $\|\langle \cdot\rangle r\|_{L^2(\mathbb{R}_k)\cap L^\infty(\mathbb{R}_k)}\leq \rho_0$ for some $\rho_0>0$, then there exists a positive constant $C(\rho_0)$ such that
		\begin{align*}
			&|M(i)-I|,\ |M(0)-I|\leq C(\rho_0)\|r\|_{L^2(\mathbb{R}_k)\cap L^\infty(\mathbb{R}_k)},\\
			&|\partial_yM(i)|,\  |\partial_yM^{i,(1)}(y)|\leq C(\rho_0)\|(\cdot)r\|_{L^2(\mathbb{R}_k)\cap L^\infty(\mathbb{R}_k)}.
		\end{align*}
	\end{Proposition}
	\begin{proof}
		Applying the H\"older inequality yields the result. For example,
		\begin{align*}
			|M(i;y)-I|&=\Big|\frac{1}{2\pi }\int_{\mathbb{R}}\frac{(\mu(s;y)-I )w(s;y)}{s-i}ds+\frac{1}{2\pi }\int_{\mathbb{R}}\frac{w(s;y)}{s-i}ds\Big|\\
			&\leq (\|r\|_\infty\|\mu-I\|_2+\|r\|_2)\|1/(\cdot-i)\|_2\\
			&\leq C\|r\|_2(2\|r\|_\infty+1)\leq C(\rho_0)\|r\|_{L^2(\mathbb{R}_k)\cap L^\infty(\mathbb{R}_k)}.
		\end{align*}
		The others can be estimated in the same way. This completes the proof of Proposition \ref{proMi}.
	\end{proof}
	The following corollary is obtained directly from the above proposition.
	\begin{corollary}\label{coroi}
		There exists a constant $\epsilon_0>0$ such that when $\|\langle \cdot\rangle r\|_{L^2(\mathbb{R}_k)\cap L^\infty(\mathbb{R}_k)}<\epsilon_0$, it follows that
		$$\Big| M_{11}(0)-1\Big|<1.$$
	\end{corollary}

	\section { The  existence of global solutions}\label{sec6}
	
	\subsection{Time evolution of reflection coefficient}\label{sect}
	
	According to the time  spectral problem in the Lax pair  \eqref{lax0}  and  scattering relation \eqref{scattering}, it is found  that the   time evolution of $a(z;t )$ and  $b(z;t )$ satisfy the equations
	\begin{align}
		\partial_t a(z;t )=0,\ 	\partial_t b(z;t )=-\dfrac{2iz(z^2-1)}{z^2+1}b(z;t),\ z\in\mathbb{R},
	\end{align}
	which yield
	\begin{align}
		a(z;t )=a(z;0 ),\ b(z;t )=e^{-\frac{2iz(z^2-1)t}{z^2+1}}b(z ;0),\ z\in\mathbb{R}.
	\end{align}
	Hence, one can  define the time-dependent reflection coefficient by
		\begin{align*}
		r(z;t )=e^{-\frac{2iz(z^2-1)t}{(z^2+1)^2}}r( z ;0 ),\ z\in\mathbb{R}.\label{ct}
	\end{align*}
	%whose  proof
	%is similar to Lemma 3.1 in  \cite{YFL}.
	And we have the estimate in the following lemma which is obtained by a simple calculation.
	\begin{lemma}\label{pror}
	When  $r(\cdot;0)\in  H^{1,1}(\mathbb{R}_k)$,  it is inferred that  $r(\cdot;t)\in  H^{1,1}(\mathbb{R}_k)$. And
		\begin{align}
			\|r(\cdot;t)\|_{H^{1,1}(\mathbb{R}_k)}\leq C(t)	\|r(\cdot;0)\|_{H^{1,1}(\mathbb{R}_k)},
		\end{align}
		where $C(t)>0$  may grow at most polynomially in $t$ but it remains finite for every $t>0$. And 	if $ \tilde{r}(\cdot;0) \in     H^{1,1}(\mathbb{R}_k)$, then for every $t\in\mathbb{R}^+$, it also leads to $ \tilde{r}(\cdot;t)= e^{-\frac{2iz(z^2-1)t}{(z^2+1)^2}}\tilde{r}( z ;0 )\in     H^{1,1}(\mathbb{R}_k)$.
	\end{lemma}
	Furthermore, we denote $M_l(z;t,y)$ as the  solution for  the  RH problem \ref{RHP1} under time-dependent reflection coefficient
	$r(\cdot;t)$, and denote $M_r(z;t,y)$ as the  solution for  the  RH problem \ref{RHP2} under time-dependent reflection coefficient
	$\tilde{r}(\cdot;t)$.
	%$M_r(z;0,y)$ also has  a similar  time evolution denoted by $M_r(z;t,y)$.
	 From Propositions  \ref{lemmaG} and  \ref{Pro3.2}, the following proposition holds immediately.
	\begin{Proposition}\label{exM}
		Assume that $r(\cdot;0),\ \tilde{r}(\cdot;0)\in H^{1,1}(\mathbb{R}_k )$. There  exist  $M_l(z;t,y)$ and  $M_r(z;t,y)$ uniquely for  $\forall t\in\mathbb{R}^+$.
	\end{Proposition}
	For convenience, we define
	\begin{align}
		M(z;t,y)=\left\{ \begin{array}{ll}
			M_l(z;t,y),   &\text{as } y\in \mathbb{R}^+,\\[12pt]
			M_r(z;t,y) , &\text{as }y\in \mathbb{R}^-,\\
		\end{array}\right.\label{M1t}
	\end{align}
	
	The crux of the matter is whether the functions  reconstituted from the time-depended RH problem via Lemma \ref{lemm3.2} is the solution of mCH equation \eqref{mch}. In the rest of this subsection, we verify  that the reconstruction formula
	remain accessible for $M_l(z;t,y)$ with $y\in\mathbb{R}^+$ and for $M_r(z;t,y)$ with $y\in\mathbb{R}^-$.
	The first step is to prove the existence of $\partial_t M(z;t,y)$ and $\partial_y M(z;t,y)$. We give the details when  $y\in\mathbb{R}^+$. Take $t$-derivative at the both side of
	\eqref{mu} and obtain
	\begin{align}
		\partial_t\mu(z;y)=\partial_t\mathcal{C}(w)(z;y)+(\partial_t\mathcal{K})(	\mu-I)(z;y)+\mathcal{K}(\partial_t	\mu)(z;y),
	\end{align}
	where
	\begin{align*}
		&\|\partial_t\mathcal{C}(w)\|_2=\Big\|\mathcal{C}\left( \frac{z(z^2-1)}{(z^2+1)^2}w\right) \Big\|_2\leq C\|r\|_2,\\
		&\|(\partial_t\mathcal{K})(	\mu-I)\|_2=\Big\|\mathcal{K}\left( \frac{z(z^2-1)}{(z^2+1)^2}(	\mu-I)\right)\Big \|_2\leq C\|r\|_\infty\|\mu-I\|_2.
	\end{align*}
	So the function $\partial_t\mu(z;y)$ exists uniquely in $L^2(\mathbb{R})$ with
	\begin{align*}
		\partial_t\mu(z;y)=(I-\mathcal{K})^{-1}\left( \partial_t\mathcal{C}(w)+(\partial_t\mathcal{K})(	\mu-I)\right)(z;y),
	\end{align*}
	which establishes the existence of $\partial_t M(z;t,y)$ with
	\begin{align}
		\partial_t M(z;t,y)=\mathcal{C}(\partial_t\mu w)(z;y)+\mathcal{C}((\mu-I )\partial_tw)(z;y)+\mathcal{C}(\partial_tw)(z;y).
	\end{align}
	The existence of $\partial_y M(z;t,y)$ is established in a similar analysis.
	Moreover, the following proposition is obtained in the similarly way as Proposition \ref{proMi},
	\begin{Proposition}\label{proMit}
		Suppose that $r\in H^{1,1}(\mathbb{R}_k)$ with $\|r\|_{H^{1,1}(\mathbb{R}_k)}\leq \rho_0$ for some $\rho_0>0$. Then there exists a positive constant $C(\rho_0)$ such that
		\begin{align*}
			|\partial_tM(i)|,\ &	|\partial_t\partial_yM(i)|,\  |\partial_tM^{i,(1)}(y)|,\ |\partial_t\partial_yM^{i,(1)}(y)|,\\
			&|\partial_t\partial_yM^{i,(1)}(y)|\leq C(\rho_0)\|r\|_{H^{1,1}(\mathbb{R}_k)}.
		\end{align*}
	\end{Proposition}
	In view of the symmetry of $M$ in \eqref{syM} and the fact that det$M$=1, we  denote
	\begin{align*}
		M(0)=\left(\begin{array}{cc}
			\beta_0  & \eta_0\\
			\eta_0	& \beta_0
		\end{array}\right),
		\hspace{0.3cm}	M(i)=\left(\begin{array}{cc}
			f_0  & \frac{\eta_0}{2f_0}\\
			\frac{\beta_0-1}{\eta_0}f_0	& \frac{\beta_0+1}{2f_0}
		\end{array}\right),
		\hspace{0.3cm}M^{i,(1)}=\left(\begin{array}{cc}
			\frac{\beta_0-1}{\eta_0}g_1  & g_2\\
			g_1	& \frac{\beta_0-1}{\eta_0}g_2
		\end{array}\right),
	\end{align*}
	where $\eta_0\in i\mathbb{R}$, $\beta_0,\  f_0,\   g_1,\  g_2\in\mathbb{R}$. Specially,  when $\eta_0=0$, $(\beta_0-1)/\eta_0=0$.
	It also follows that $M(i)\sigma_3 M(i)^{-1}=\sigma_3M(0)^{-1}$.
	We recall the symbol that
	\begin{align*}
		\eta=\lim_{z\to \infty}z(M(z)-I)_{12}.
	\end{align*}
	In addition, define
	\begin{align}
		\Psi(z;y,t)=M(z;y,t)e^{-\frac{i}{4}(z-1/z)y\sigma_3+2i\frac{z(z^2-1)}{(z^2+1)^2}t\sigma_3},\label{psi}
	\end{align}
	which admits the proposition below.
	\begin{Proposition}\label{pro lax}
		$\Psi$ defined in \eqref{psi} satisfies the differential equations
		\begin{align}
			\Psi_y=A\Psi,\hspace*{0.5cm}	\Psi_t=B\Psi,
		\end{align}
		where
		\begin{align*}
			A=&-\frac{iz}{4}\sigma_3+\frac{i}{2}\eta\sigma_1+\frac{i}{4z}(\beta_0^2+\eta_0^2)\sigma_3+\frac{1}{2z}\eta_0\beta_0\sigma_2,\\
			B=&-\frac{1}{(z-i)^2}\left(\begin{array}{cc}
				\beta_0  & -\eta_0\\
				\eta_0	& -\beta_0
			\end{array}\right)+\frac{i}{z-i}\left(\begin{array}{cc}
				\beta_0  & -\eta_0\\
				\eta_0	& -\beta_0
			\end{array}\right)\\
			&-\frac{1}{z-i}\left(\begin{array}{cc}
				2(\frac{\beta_0-1}{\eta_0}g_2f_0+ \frac{\eta_0}{2f_0} g_1)& -2f_0g_2- 2\frac{\beta_0-1}{2f_0} g_1\\
				2\frac{\beta_0-1}{\beta_0+1} g_2f_0+\frac{\beta_0+1}{f_0} g_1	& -2(\frac{\beta_0-1}{\eta_0}g_2f_0+ \frac{\eta_0}{2f_0} g_1)
			\end{array}\right)\\
			&+\frac{1}{(z+i)^2}\left(\begin{array}{cc}
				\beta_0  & -\eta_0\\
				\eta_0	& -\beta_0
			\end{array}\right)+\frac{i}{z+i}\left(\begin{array}{cc}
				\beta_0  & -\eta_0\\
				\eta_0	& -\beta_0
			\end{array}\right)\\
			&+\frac{1}{z+i}\left(\begin{array}{cc}
				-2(\frac{\beta_0-1}{\eta_0}g_2f_0+ \frac{\eta_0}{2f_0} g_1)& 2\frac{\beta_0-1}{\beta_0+1} g_2f_0+\frac{\beta_0+1}{f_0} g_1\\
				-2f_0g_2- 2\frac{\beta_0-1}{2f_0} g_1	& 2(\frac{\beta_0-1}{\eta_0}g_2f_0+ \frac{\eta_0}{2f_0} g_1)
			\end{array}\right)
		\end{align*}
		Moreover, via denoting
		\begin{align*}
			&\tilde{u}=-\frac{g_1}{f_0}-2\frac{\beta_0-1}{\eta_0^2}f_0g_2,\hspace{0.5cm}\\
			&\tilde{q}=\frac{1}{\beta_0},\hspace{0.5cm}\tilde{m}=\frac{\eta_0}{\beta_0i},
		\end{align*}
		the compatibility condition
		\begin{align*}
			A_t+AB-B_y-BA=0
		\end{align*}
		yields the mCH equation in the $(y,t)$ variables. In addition,  combining the variable transformation in \eqref{recons x}, we obtain
		\begin{align}
			x_y=\beta_0,\label{dxy}\hspace{0.5cm}\tilde{u}_x=\frac{g_1}{f_0}-2\frac{\beta_0-1}{\eta_0^2}f_0g_2,
		\end{align}
	\end{Proposition}
	\begin{proof}
		The definition of $\Psi$ implies that the jump of $\Psi$ is  independent of $y$ and $t$.  Consequently, $\Psi_y\Psi^{-1}$ and $\Psi_t\Psi^{-1}$ have no jump. We analyze $\Psi_y\Psi^{-1}$ first. \eqref{psi} gives that
		\begin{align*}
			\Psi_y\Psi^{-1}=M_yM^{-1}-\frac{i}{4}\left(z-\frac{1}{z} \right) M\sigma_3M^{-1},
		\end{align*}
		which  is a meromorphic function, with possible singularities at $z=0$ and $z=\infty$. As $z\to\infty$,
		\begin{align*}
			\Psi_y\Psi^{-1}=-\frac{i}{4}z\sigma_3+\frac{i}{2}\eta\sigma_1+\mathcal{O}(1/z),
		\end{align*}
		while as $z\to0$,
		\begin{align*}
			\Psi_y\Psi^{-1}=\frac{i}{4z}M(0)\sigma_3M(0)^{-1}+\mathcal{O}(1).
		\end{align*}
		Therefore, the function
		\begin{align*}
			\Psi_y\Psi^{-1}-\frac{i}{4z}M(0)\sigma_3M(0)^{-1}+ \frac{i}{4}z\sigma_3-\frac{i}{2}\eta\sigma_1
		\end{align*}
		is  holomorphic in $\mathbb{C}$ and vanish at $z\to\infty$. Then, by Liouville's theorem, it vanishes identically, which leads to the result $ \Psi_y=A\Psi$. Similarly,  	\begin{align*}
			\Psi_t\Psi^{-1}=M_tM^{-1}+2i\frac{z(z^2-1)}{(z^2+1)^2} M\sigma_3M^{-1},
		\end{align*}
		is a meromorphic function, with possible singularities at $z=\pm i$. From the decomposition
		\begin{align*}
			2i\frac{z(z^2-1)}{(z^2+1)^2}=\frac{i}{z+i}+\frac{i}{z-i}+\frac{1}{(z+i)^2}-\frac{1}{(z-i)^2},
		\end{align*}
		we obtain that $ \Psi_t=B\Psi$. Using the  compatibility condition for the function $\Psi$ results in the compatibility equation: $A_t + AB - B_y - BA = 0$.
Considering this compatibility equation at the singular points of both $A$ and $B$, we can derive the algebraic and differential relationships among the coefficients of $A$ and $B$. These relationships can be further reduced to the  equation \eqref{mchy}. Moreover, when we combine the compatibility equation with the expression for the change of variables as presented in equation \eqref{recons x}, it leads to the mCH equation \eqref{mch}. This completes the proof of Proposition \ref{pro lax}.
	\end{proof}

%Using this compatibility condition on $\Psi$ then yields the compatibility equation $A_t+AB-B_y-BA=0.$
%In view of  this compatibility equation at the singular points for $A$ and $B$, the algebraic and differential equations among
%		the coefficients of $A$ and $B$ are derived,  which can be reduced to \eqref{mchy}. Furthermore, the  compatibility equation together with the expression for the %change of variables in \eqref{recons x} leads to  the mCH equation \eqref{mch}. This completes the proof of Proposition \ref{pro lax}.
	%\end{proof}

	\subsection {The proof of main  result}\label{sec62}
	
	The main goal of this subsection is to give the proof of Theorem \ref{last}.
	In view of the result in Section \ref{sec3}, we  use $M_l(z;t,y)$ to recover the solution $m(t,x(y,t))$ for  $y\in\mathbb{R}^+$, and   $M_r(z;t,y)$ to recover the solution $m(t,x(y,t))$ for $y\in\mathbb{R}^-$ through Lemma \ref{lemm3.2}.
	The first step is to give the proof of the  lemma below.
	\begin{lemma}\label{the1}
		Assume that  the initial data   $ m_0 \in H^{2,1}(\mathbb{R} )$ such that $\|m_0\|_{H^{2,1}}$ is small.
		Then  there exists a unique global solution $m(t,y)$ for every $t\in\mathbb{R}^+$ to  the Cauchy  problem \eqref{mch}-\eqref{mch1}  for   the mCH equation in $C([0, +\infty); L^{\infty}(\mathbb{R}))$.
	\end{lemma}
	\begin{proof}
		When $ m_0 \in H^{2,1}(\mathbb{R} )$, it follows from Lemma  \ref{lemmasec2}  that $r(z;0)\in H^{1,1}(\mathbb{R}_k)$.
		%As  shown in Proposition \ref{exM}, for every $t>0$, $M(z;t,y)$ defined in \eqref{M1t} exists uniquely for every $t\in\mathbb{R}^+$, $y\in\mathbb{R}$.
		As  shown in Proposition \ref{exM}, there exists a unique solution $M(z;t,y)$ defined in \eqref{M1t} for every $t\in\mathbb{R}^+$, $y\in\mathbb{R}$.
		
		As shown in Theorem 4.1  \cite{qu7},  there exists a time $T>0$ such that  the initial-value problem \eqref{mch} has a unique solution  $m(t,\cdot)\in C([0,T],H^2(\mathbb{R}))$. Hence,  $m$ reconstituted from $M(z;t,y)$ in Lemma \ref{lemm3.2} must be unique.
		It remains to prove $m(t,\cdot)\in C([0, +\infty); L^{\infty}(\mathbb{R}))$. %which is equivalent to prove that  $q(t,\cdot)=\sqrt{1+m(t,\cdot)^2}\in C([0, +\infty); L^{\infty}(\mathbb{R}))$.
		It suffices to  show that  $q(t,\cdot)=\sqrt{1+m(t,\cdot)^2}\in C([0, +\infty); L^{\infty}(\mathbb{R}))$.
		
		From Lemma \ref{pror}, for ever $t\in[0,T]$,
		$\|r(\cdot;t)\|_{L^2(\mathbb{R}_k)\cap L^\infty(\mathbb{R}_k)}$ is independent of $t$.
		The symmetry of $M(z;t,y)$ in \eqref{syM} implies that
		\begin{align*}
			M_{11}(0;t,y)=M_{22}(0;t,y)\in\mathbb{R},\hspace{0.5cm}	M_{12}(0;t,y)=M_{21}(0;t,y)i\in\mathbb{R}
		\end{align*}
		Noting that det$M\equiv1$, thus $|M_{11}(0;t,y)|\leq1$.
		As shown in Corollary \ref{coroi}, when $\|r(\cdot;0)\|_{L^2(\mathbb{R}_k)\cap L^\infty(\mathbb{R}_k)}<\epsilon_0$., it holds that for every $t\in[0,T]$,
		$$\Big	|M_{11}(0;t,y)-1\Big|<1.$$
		Therefore, from \eqref{recons q}, it appears that
		\begin{align}
			1\leq|q(t,y)|<(1-|M_{11}(0;t,y)-1|)^{-1}<\infty.\label{|q|}
		\end{align}
		If there exists a maximal existence time $T_{max}>0$ such that $q(t,\cdot)$ only exists in $C([0, T_{max}]; L^{\infty}(\mathbb{R}))$, then it must appears that $\lim_{t\to T_{max}-}\|q(t,\cdot)\|_\infty=\infty$, from which we derive a contradiction.
		This final argument then completes  the proof of Lemma \ref{the1}.
\end{proof}		
		%Combining Proposition \ref{proMi} and \ref{proMit}, we finally  arrive at the  continuity of $m(t,y)$ about $t$. Then we complete the proof.

	Because the estimates in the above section are under the scaling transformation of $y$. The  equivalency between the integral norm in $\mathbb{R}_x$ and  $\mathbb{R}_y$ is first given to arrive at the $L^2$-estimates under the variable $x$.
	
	\begin{lemma}\label{lemmaytox}
		Assume that $\|r(\cdot;0)\|_{L^2(\mathbb{R}_k)\cap L^\infty(\mathbb{R}_k)}<\epsilon_0$. Then
		for any function $h(x(\cdot))\in L^p(\mathbb{R}_y)$ with $x(y)$ is given in \eqref{recons x}, $1\leq p\leq \infty$, it follows that $h\in L^p(\mathbb{R}_x)$. Furthermore, it is also inferred that
		\begin{align*}
			C_1(\|r(\cdot;0)\|_{L^2(\mathbb{R}_k)\cap L^\infty(\mathbb{R}_k)})|y|\leq |x|\leq	C_2(\|r(\cdot;0)\|_{L^2(\mathbb{R}_k)\cap L^\infty(\mathbb{R}_k)})|y|.
		\end{align*}
		Therefore, $h\in L^{p,s}(\mathbb{R}_x)$  is  equivalent to $h\in L^{p,s}(\mathbb{R}_y)$ for any $s>0$.
	\end{lemma}
	\begin{proof}
		The  equivalency between $ L^p(\mathbb{R}_x)$ and $ L^p(\mathbb{R}_y)$  is obtained directly from \eqref{dxy} and \eqref{|q|} under $\|r(\cdot;0)\|_{L^2(\mathbb{R}_k)\cap L^\infty(\mathbb{R}_k)}<\epsilon_0$.
		As shown in Lemma \ref{pror}, for ever $t\in\mathbb{R}^+$,
		$\|r(\cdot;t)\|_{L^2(\mathbb{R}_k)\cap L^\infty(\mathbb{R}_k)}$ is independent of $t$.
		Recall the symmetry of $M(z)$ that
		\begin{align*}
			M(z)=\sigma_3\overline{	M(-\bar{z})}\sigma_3.
		\end{align*}
		This implies that $M_{11}(i)$, $M_{22}(i)\in\mathbb{R}$ and  $M_{12}(i)$, $M_{21}(i)\in i\mathbb{R}$.
		Noting that det$M\equiv1$, then the boundedness of Proposition \ref{proMi} gives that the functions $|M_{11}(i)+M_{21}(i)|$ and $|M_{12}(i)+M_{22}(i)|$ have nonzero infimum for every $t\in \mathbb{R}^+,\ y\in\mathbb{R}$. In fact,  taking $|M_{11}(i)+M_{21}(i)|$, $y\in\mathbb{R}^+$, as an example,
		\begin{align*}
			1&=\big|M_{22}(i)M_{11}(i)+\text{Im}\left( M_{12}(i)M_{21}(i)\right) \big|\\
			&\leq C |M(i)||M_{11}(i)+M_{21}(i)|,
		\end{align*}
		which leads to that
		\begin{align*}
			|M_{11}(i)+M_{21}(i)|\geq 1/C(\|r(\cdot;0)\|_{L^2(\mathbb{R}_k)\cap L^\infty(\mathbb{R}_k)})>0.
		\end{align*}
		Thus,  it is deduced from \eqref{recons x} that
		\begin{align*}
			|x-y|\leq C(\|r(\cdot;0)\|_{L^2(\mathbb{R}_k)\cap L^\infty(\mathbb{R}_k)}).
		\end{align*}
		This concludes  the result in Lemma \ref{lemmaytox}.
	\end{proof}
	\begin{remark}
		In fact, it follows from the definition of $y$ in \eqref{transy} that
		\begin{align*}
			|x-y|\leq \int_{\mathbb{R}}(\sqrt{m^2(t,s)+1}-1)ds,
		\end{align*}
		where $\int_{\mathbb{R}}(\sqrt{m^2(t,s)+1}-1)ds$ is a conserved quantity  independent of $t$. Thus, $|x-y|$ can be controlled by $\|m_0\|_{2}$.
		\end{remark}
	
	The following  result from the Fourier theory is useful to prove Theorem \ref{last}.
	\begin{lemma}\label{lemmaFourier}
		If $f(z)$ is a complex valued function defined on $\mathbb{R}$ and admits the symmetry $f(z)=f(-1/z)$, then $f(z(k))$ is well-defined under the variable $k=z-1/z$. Denote $\tilde{f}(k)=f(z(k))$. \\
		(a) If $f\in H^{1,1}(\mathbb{R}_k)$, then
\begin{align*}
			\int_{\mathbb{R}}f(z)e^{-\frac{i}{2}(z-1/z)y}dz=\int_{\mathbb{R}}\tilde{f}(k)e^{-\frac{i}{2}ky}dk,
		\end{align*}
		which implies that
		\begin{align*}
			\int_{\mathbb{R}}f(z)e^{-\frac{i}{2}(z-1/z)y}dz\in H^{1,1}(\mathbb{R}_y).
		\end{align*}
		(b) Denote  $\mathcal{C}^k$   as the Cauchy operator on $k$-plane. Then $\mathcal{C}_\pm(f)$ is  also well-defined under the variable $k$ and
		\begin{align}
			\mathcal{C}_\pm(f)(z)=\mathcal{C}^k_\pm(f)(k),
		\end{align}
		which yields that  $\|\mathcal{C}_\pm(f)\|_{L^2(\mathbb{R}_z)}=\|\mathcal{C}^k_\pm(f)\|_{L^2(\mathbb{R}_k)}$.\\
		(c) If $g$ also admits that $g(z)=g(-1/z)$ and $\tilde{g}(k)=g(z(k))$ is in $ H^{1}(\mathbb{R}_k)$, then
		\begin{align*}
			\int_{\mathbb{R}^+} \|\mathcal{C}_+(fe^{2\theta})\|_2^2\|\mathcal{C}_\pm(ge^{\pm2\theta})\|_2^2dy\leq C \|\tilde{f}\|^2_{H^{1/4}(\mathbb{R}_k)}\|\tilde{g}\|^2_{H^{1/4}(\mathbb{R}_k)},\\
			\int_{\mathbb{R}^+}y^2 \|\mathcal{C}_+(fe^{2\theta})\|_2^2\|\mathcal{C}_\pm(ge^{\pm2\theta})\|_2^2dy\leq C \|\tilde{f}\|^2_{H^{3/4}(\mathbb{R}_k)}\|\tilde{g}\|_{H^{3/4}(\mathbb{R}_k)}^2.
		\end{align*}
	\end{lemma}
	\begin{proof}
		Recall that for $k=z-1/z$, the map $k\to z$ is given by
		\begin{align*}
			&z_+(k)=(k+\sqrt{k^2+4})/2:\ \mathbb{R}\to\mathbb{R}^+,\\ &z_-(k)=(k-\sqrt{k^2+4})/2:\ \mathbb{R}\to\mathbb{R}^-.
		\end{align*}
		Therefore, it is readily seen that
		\begin{align*}
			\int_{\mathbb{R}}f(z)e^{-\frac{i}{2}(z-1/z)y}dz&=\int_{\mathbb{R}}\tilde{f}(k)e^{-\frac{i}{2}ky}\frac{1}{1+z_+(k)^{-2}}dk+\int_{\mathbb{R}}\tilde{f}(k)e^{-\frac{i}{2}ky}\frac{1}{1+z_-(k)^{-2}}dk\\
			&=\int_{\mathbb{R}}\tilde{f}(k)e^{-\frac{i}{2}ky}dk.
		\end{align*}
		And the assertion in (b) is obtained from the fact that
		\begin{align*}
			\frac{1}{s-z}+\dfrac{1}{s^2}\cdot\frac{1}{-\frac{1}{s}-z}=\frac{1+\frac{1}{s^2}}{s-\frac{1}{s}-z+\frac{1}{z}}.
		\end{align*}
		For the claim in (c), we now give the details of the calculation for $\|\mathcal{C}_+(fe^{2\theta})\|_2^2\|\mathcal{C}_+(ge^{2\theta})\|_2^2$. It is found  that
		\begin{align*}
			\int_{\mathbb{R}^+}& \|\mathcal{C}_+(fe^{2\theta})\|_2^2\|\mathcal{C}_+(ge^{2\theta})\|_2^2dy=\int_{\mathbb{R}^+}  \|\mathcal{C}^k_+(fe^{2\theta})\|_{L^2(\mathbb{R}_k)}^2\|\mathcal{C}^k_+(ge^{2\theta})\|_{L^2(\mathbb{R}_k)}^2dy\\
			&=\int_{\mathbb{R}^+}\int_{\mathbb{R}}X_+(k)|\hat{\tilde{f}}(k+y/4\pi)|^2dk\int_{\mathbb{R}}X_+(k)|\hat{\tilde{g}}(k+y/4\pi)|^2dkdy\\
			&=\int_{\mathbb{R}^+}\int_{y/4\pi}^\infty|\hat{\tilde{f}}(k)|^2dk\int_{y/4\pi}^\infty|\hat{\tilde{g}}(k)|^2dkdy\\
			&\leq \left( \int_{\mathbb{R}^+}(\int_{y/4\pi}^\infty|\hat{\tilde{f}}(k)|^2dk)^2dy\right) ^{1/2}\left( \int_{\mathbb{R}^+}(\int_{y/4\pi}^\infty|\hat{\tilde{g}}(k)|^2dk)^2dy\right) ^{1/2}\\
			&\leq \int_{\mathbb{R}^+} ( \int_0^{k/4\pi}|\hat{\tilde{f}}(k)|^4dy) ^{1/2}dk\int_{\mathbb{R}^+} ( \int_0^{k/4\pi}|\hat{\tilde{g}}(k)|^4dy) ^{1/2}dk\\
			&=\int_{\mathbb{R}^+} |\hat{\tilde{f}}(k)|^2(k/4\pi) ^{1/2}dk\int_{\mathbb{R}^+} |\hat{\tilde{g}}(k)|^2(k/4\pi) ^{1/2}dk\leq C\|\hat{\tilde{f}}\|_{{2,1/4}(\mathbb{R}_k)}^2\|\hat{\tilde{g}}\|_{{2,1/4}(\mathbb{R}_k)}^2,
		\end{align*}
		and
		\begin{align*}
			\int_{\mathbb{R}^+}y^2& \|\mathcal{C}_+(fe^{2\theta})\|_2^2\|\mathcal{C}_\pm(ge^{\pm2\theta})\|_2^2dy=\int_{\mathbb{R}^+}\int_{y/4\pi}^\infty y|\hat{\tilde{f}}(k)|^2dk\int_{y/4\pi}^\infty y|\hat{\tilde{g}}(k)|^2dkdy\\
			&\leq \left( \int_{\mathbb{R}^+}(\int_{y/4\pi}^\infty y|\hat{\tilde{f}}(k)|^2dk)^2dy\right) ^{1/2}\left( \int_{\mathbb{R}^+}(\int_{y/4\pi}^\infty y|\hat{\tilde{g}}(k)|^2dk)^2dy\right) ^{1/2}\\
			&\leq \int_{\mathbb{R}^+} ( \int_0^{k/4\pi}y^2|\hat{\tilde{f}}(k)|^4dy) ^{1/2}dk\int_{\mathbb{R}^+} ( \int_0^{k/4\pi}y^2|\hat{\tilde{g}}(k)|^4dy) ^{1/2}dk\\
			&=\frac{1}{3}\int_{\mathbb{R}^+} |\hat{\tilde{f}}(k)|^2(k/4\pi) ^{3/2}dk\int_{\mathbb{R}^+} |\hat{\tilde{g}}(k)|^2(k/4\pi) ^{3/2}dk\leq C\|\hat{\tilde{f}}\|_{{2,3/4}(\mathbb{R}_k)}^2\|\hat{\tilde{g}}\|_{{2,3/4}(\mathbb{R}_k)}^2.
		\end{align*}
	This completes the proof of Lemma \ref {lemmaFourier}.
	\end{proof}
	We shall now recover $m$  from the RH problem $M$ and give
	$L^2$-estimate of $m$ and establish our main result in Theorem \ref{last}.
	\begin {proof} [Proof of  Theorem \ref{last}]
	We only give the details of the proof of $m(t,\cdot)\in H^{2,1}(\mathbb{R}^+_y)$.
	Together with  Lemma \ref{lemmaytox} and the fact that $|\partial_xy|$,  $|\partial_x^2y|$ is bounded as shown in Lemma \ref{the1}, it follows that $m(t,\cdot)\in H^{2,1}(\mathbb{R}^+_x)$.
	Recall the reconstruction formula in Lemma \ref{lemm3.2}:
	\begin{align}
		\eta=\lim_{z\to \infty}zM_{12}.
	\end{align}
	From \eqref{intM}, it is adduced that
	\begin{align}
		\eta=\frac{i}{2\pi}\int_{\mathbb{R}}(\mu_{11}-1)\bar{r}e^{2\theta}dz+\frac{i}{2\pi}\int_{\mathbb{R}}\bar{r}e^{2\theta}dz.\label{eta}
	\end{align}
	For convenience, denote
	\begin{align}
		\eta_1=\int_{\mathbb{R}}(\mu_{11}-1)\bar{r}e^{2\theta}dz,\ \eta_2=\int_{\mathbb{R}}\bar{r}e^{2\theta}dz.
	\end{align}
	Consequently, Lemma \ref{lemmaFourier} gives that $\eta_2\in H^{1,1}(\mathbb{R}_y)$. And for $\eta_1$, invoking \eqref{mu}, it holds that
	\begin{align}
		\eta_1=\int_{\mathbb{R}}\mathcal{C}_-(re^{-2\theta}\mu_{12})\bar{r}e^{2\theta}dz=-\int_{\mathbb{R}}\mathcal{C}_+(\bar{r}e^{2\theta})re^{-2\theta}\mu_{12}dz,\label{eta1}
	\end{align}
	from which and the estimate in \eqref{muL2} we obtain that
	\begin{align*}
		|\eta_1|\leq \|\mathcal{C}_+(\bar{r}e^{2\theta})\|_2\|\mu_{12}\|_2\|r\|_\infty\leq2\|\mathcal{C}_+(\bar{r}e^{2\theta})\|_2\|r\|_\infty(\|\mathcal{C}_+(\bar{r}e^{2\theta})\|_2+\|\mathcal{C}_-(re^{-2\theta})\|_2).
	\end{align*}
	For the term $\|\mathcal{C}_+(\bar{r}e^{2\theta})\|_2^2$ in the right of above inequality, applying Lemma \ref{lemmaFourier} yields
	\begin{align*}
		\int_{\mathbb{R}^+} \|\mathcal{C}_+(\bar{r}e^{2\theta})\|_2^4dy\leq C\|\hat{r}\|_{{2,1/4}(\mathbb{R}_k)}^4,
	\end{align*}
	and
	\begin{align*}
		\int_{\mathbb{R}^+} \big|y\big|^2\|\mathcal{C}_+(\bar{r}e^{2\theta})\|_2^4dy& \leq C\|\hat{r}\|_{{2,3/4}(\mathbb{R}_k)}^4.
	\end{align*}
	Similar calculation of $ \|\mathcal{C}_+(\bar{r}e^{2\theta})\|_2\|\mathcal{C}_-(re^{-2\theta})\|_2$ gives that  $\eta_1\in L^{2,1}(\mathbb{R}^+_y)$. Take the $y$-derivative in \eqref{eta1} and obtain
	\begin{align*}
		\partial_y\eta_1=&-\int_{\mathbb{R}}\mathcal{C}_+(\frac{i}{2}k(\cdot)\bar{r}e^{2\theta})re^{-2\theta}\mu_{12}dz-\int_{\mathbb{R}}\mathcal{C}_+(\bar{r}e^{2\theta})\frac{i}{2}k(z)re^{-2\theta}\mu_{12}dz\\
		&-\int_{\mathbb{R}}\mathcal{C}_+(\bar{r}e^{2\theta})re^{-2\theta}\partial_y\mu_{12}dz,
	\end{align*}
	which implies that
	\begin{align*}
		|\partial_y\eta_1|\leq&\|r\|_\infty\left(  \|\mathcal{C}_+(k(\cdot)\bar{r}e^{2\theta})\|_2\|\mu_{12}\|_2+\|\mathcal{C}_+(\bar{r}e^{2\theta})\|_2\|\mu_{12}\|_2+\|\mathcal{C}_+(\bar{r}e^{2\theta})\|_2\|\partial_y\mu_{12}\|_2\right) \\
		\leq& 2\|r\|_\infty(\|\mathcal{C}_+(\bar{r}e^{2\theta})\|_2+\|\mathcal{C}_-(re^{-2\theta})\|_2)\left( \|\mathcal{C}_+(k(\cdot)\bar{r}e^{2\theta})\|_2+\|\mathcal{C}_+(\bar{r}e^{2\theta})\|_2\right) \\
		&+2\|r\|_\infty\|\mathcal{C}_+(\bar{r}e^{2\theta})\|_2(\|\mathcal{C}_+(k(\cdot)\bar{r}e^{2\theta})\|_2+\|\mathcal{C}_-(k(\cdot)re^{-2\theta})\|_2)\nonumber\\
		&+8\|r\|_\infty\|\mathcal{C}_+(\bar{r}e^{2\theta}))\|_2(\|\mathcal{C}_+(\bar{r}e^{2\theta})\|_2+\|\mathcal{C}_-(re^{-2\theta})\|_2)\|k(\cdot)r\|_{{\infty}}
	\end{align*}
	The last inequality is obtained from \eqref{muyL2}.
	We give the details of the estimate of the term $\|\mathcal{C}_+(\bar{r}e^{2\theta})\|_2\|\mathcal{C}_+(k(\cdot)\bar{r}e^{2\theta})\|_2$. Similar claims can be applied to the others.
	
	Via  Lemma \ref{lemmaFourier}, it holds that
	\begin{align*}
		&\int_{\mathbb{R}^+} \|\mathcal{C}_+(\bar{r}e^{2\theta})\|_2^2\|\mathcal{C}_+(k(\cdot)\bar{r}e^{2\theta})\|_2^2dy\leq C\|(\cdot)r\|_{H^{1/4}(\mathbb{R}_k)}^2\|r\|_{H^{1/4}(\mathbb{R}_k)}^2,
	\end{align*}
	and
	\begin{align*}
		&\int_{\mathbb{R}^+} y^2\|\mathcal{C}_+(\bar{r}e^{2\theta})\|_2^2\|\mathcal{C}_+(k(\cdot)\bar{r}e^{2\theta})\|_2^2dy\leq C \|r\|_{H^{3/4}(\mathbb{R}_k)}^2\|(\cdot)r\|_{H^{3/4}(\mathbb{R}_k)}^2.
	\end{align*}
	This in turn implies that  $	|\partial_y\eta_1|\in L^{2,1}(\mathbb{R}_y)$.
	Therefore, we finally conclude that $\eta\in H^{1,1}(\mathbb{R}^+_y)$. By \eqref{eta}, the map:
	$$r\in H^{1,1}(\mathbb{R}_k)\longrightarrow\eta\in H^{1,1}(\mathbb{R}^+_y)$$
	is locally Lipschitz continuous. Indeed, if there exists
	 another $r_1\in H^{1,1}(\mathbb{R}_k)$ with
	a solution $\mu(r_1)$ of corresponding
	equation \eqref{mu} and $\eta(r_1)$ from  \eqref{eta}, it is inferred that
	\begin{align*}
		\eta(r_1)-\eta(r)=&\frac{i}{2\pi}\int_{\mathbb{R}}(\mu_{11}(r_1)-\mu_{11}(r))\bar{r}_1e^{2\theta}dz+\frac{i}{2\pi}\int_{\mathbb{R}}(\mu_{11}(r)-1)(\bar{r}_1-\bar{r})e^{2\theta}dz\\
		&+\frac{i}{2\pi}\int_{\mathbb{R}}(\bar{r}_1-\bar{r})e^{2\theta}dz.
	\end{align*}
	Then by a similar way in the above expressions, it follows that $$\|\eta(r_1)-\eta(r)\|_{ H^{1,1}(\mathbb{R}^+_y)}\leq C (\max \{\|r\|_{H^{1,1}(\mathbb{R}_k)}, \|r_1\|_{H^{1,1}(\mathbb{R}_k)}\})(\|r\|_{H^{1,1}(\mathbb{R}_k)}-\|r_1\|_{H^{1,1}(\mathbb{R}_k)}).$$
	On the other hand, consider
	\begin{align}
		\zeta^{(+)}=-\lim_{z\to \infty}z(M_{22}-1).
	\end{align}
	Combining \eqref{intM} and \eqref{mu}, it is adduced that
	\begin{align}
		\zeta^{(+)}=\frac{1}{2\pi i}\int_{\mathbb{R}}\mathcal{C}_-(re^{-2\theta})\bar{r}e^{2\theta}ds-\frac{1}{2\pi i}\int_{\mathbb{R}}(\mu_{22}-1)re^{-2\theta}\mathcal{C}_+(\bar{r}e^{2\theta})ds.
	\end{align}
	An analogical calculation gives that $\zeta^{(+)}\in H^{1,1}(\mathbb{R}^+_y)$ and the map:
	$$r\in H^{1,1}(\mathbb{R}_k)\longrightarrow \zeta^{(+)}\in H^{1,1}(\mathbb{R}^+_y)$$
	is locally Lipschitz continuous.  Recall that
	\begin{align*}
		\eta=\dfrac{m_{x}}{(1+m^2)^{3/2}},\ \zeta^{(+)}=\frac{i}{2}\int_{y}^{+\infty}\left( \frac{m_{x}^2}{q^6}+\frac{m^2}{q^2}\right) ds.
	\end{align*}
	Then by the boundedness of $m$ and Sobolev-Gagliardo-Nirenberg inequality, we finally obtain $m(t,\cdot)\in H^{2,1}(\mathbb{R}^+_y)$ and the Lipschitz continuity from $r\in H^{1,1}(\mathbb{R}_k)$ to $ m(t,\cdot)\in H^{1,1}(\mathbb{R}^+_y)$.

	For $y\in \mathbb{R}^-$, the estimates are same.
	This gives the desired result in   Theorem \ref{last}.
	\end {proof}

	\noindent\textbf{Acknowledgements}
	
	The work of Fan and Yang is partially  supported by  the National Science
	Foundation of China under grants 12271104, 51879045 and 12247182 and China Postdoctoral Science Foundation. The work of Liu is partially supported by the Simons Foundation under grant 499875.

	\hspace*{\parindent}
	\\
	
\end{document}